\documentclass[twoside, 11pt, leqno]{amsart}
 \usepackage{amsmath, amssymb, amsfonts, amsthm, amscd} \usepackage{mathrsfs}
 \usepackage{graphicx}
 \usepackage{cancel}
 \usepackage[margin=1.4in]{geometry}

\theoremstyle{plain}
\numberwithin{equation}{section}

\newtheorem{theorem}{Theorem}[section]
\newtheorem{maintheorem}{Theorem}
\newtheorem{lemma}[theorem]{Lemma}

\newtheorem{proposition}[theorem]{Proposition}
\newtheorem{definition}[theorem]{Definition}
\newtheorem{remark}[theorem]{Remark}

\newtheorem{conditions}[theorem]{Conditions}

\newcommand{\fracsm}[2]{\begin{matrix}\frac{#1}{#2}\end{matrix}}

\newcommand{\beq}{\begin{equation}}
\newcommand{\eeq}{\end{equation}}

\newcommand{\Reals}{\mathbb{R}}

\newcommand{\sbullet}{\,\begin{picture}(1,1)(0,-3)\circle*{2}\end{picture}\ }

\DeclareMathOperator{\Ric}{Ric}

\DeclareMathOperator{\grad}{\nabla}

\begin{document}

\author{Niels Martin M\o{}ller}
\address{Niels Martin M\o{}ller, Fine Hall, Princeton University, NJ 08540.}
\email{moller@math.princeton.edu}
\title{Closed self-shrinking surfaces \\in $\mathbb{R}^3$ via the torus}
\keywords{Mean curvature flow, self-shrinkers, self-similarity, solitons, minimal surfaces, gluing constructions, stability theory, geodesics, spectral theory.}
\thanks{The author was supported  partially by NSF award DMS-1311795.}

\begin{abstract}
We construct many closed, embedded mean curvature self-shrinking surfaces $\Sigma_g^2\subseteq\Reals^3$ of high genus $g=2k$, $k\in \mathbb{N}$.

Each of these shrinking solitons has isometry group equal to a dihedral group on $2g$ elements, and comes from the "gluing", i.e. desingularizing the singular union, of the two known closed embedded self-shrinkers in $\Reals^3$: The round 2-sphere $\mathbb{S}^2$, and Angenent's self-shrinking 2-torus of revolution $\mathbb{T}^2$. This uses the results and methods N. Kapouleas developed for minimal surfaces in \cite{Ka97}--\cite{Ka}.

\end{abstract}

\date{September 29, 2014}
\maketitle

\section{Introduction}
Recall that a smooth surface $\Sigma^{2}\subseteq\Reals^{3}$ is a mean curvature self-shrinker if
it satisfies the corresponding nonlinear elliptic self-shrinking soliton PDE:
\beq\label{SSPDE}
H(\Sigma)=\frac{\langle X,\nu_{\Sigma}\rangle}{2},\quad X\in\Sigma^2\subseteq\Reals^{3},
\eeq
where $H$ denotes the mean curvature, $X$ the position vector and  $\nu_\Sigma$ is a unit length vector field normal to $\Sigma^2\subseteq\Reals^3$.

While important as singularity models in mean curvature flow (see e.g. \cite{Hu90}-\cite{Hu93} and \cite{CM6}-\cite{CM8}), the list of known closed, embedded surfaces satisfying Equation (\ref{SSPDE}) is short:
\begin{itemize}
\item[\sbullet] The round 2-sphere $\mathbb{S}^2\subseteq\Reals^3$,
\item[\sbullet] Angenent's (non-circular) 2-torus, in \cite{Ang}.
\end{itemize}
Apart from these, there is numerical evidence for the existence of a self-shrinking "fattened wire cube" in \cite{Ch}, and in higher dimensions there are Lagrangian examples (generalizing the Abresch-Langer curves \cite{AL}), found by Anciaux in \cite{Anc}.

One can rigorously construct closed, embedded, smooth mean curvature self-shrinkers with high genus $g$, embedded in Euclidean space $\Reals^3$:
\begin{maintheorem}\label{Thm:torusphere}
For every large enough even integer $g=2k$, $k\in\mathbb{N}$, there exists a compact, embedded, orientable, smooth surface without boundary $\Sigma^2_g\subseteq \Reals^3$, with the properties:
\begin{itemize}
\item[(i)] $\Sigma_g$ is a mean curvature self-shrinker of genus $g$.
\item[(ii)] $\Sigma_g$ is invariant under the dihedral symmetry group with $2g=4k$ elements.
\item[(iii)] The sequence $\{\Sigma_g\}$ converges in Hausdorff sense to the union $\mathbb{S}^2\cup\mathbb{T}^2$, where $\mathbb{T}^2$ is a rotationally symmetric self-shrinking torus in $\Reals^3$. The convergence is locally smooth away from the two intersection circles constituting $\mathbb{S}^2\cap\mathbb{T}^2$.
\end{itemize}
\end{maintheorem}

This paper consists of: (1) Brief account of the construction and its important components, and (2) Proofs of the central explicit estimates of functions that for some small $\delta, \varepsilon>0$ in an appropriate sense are $\delta$-close to being eigenfunctions of the stability operator (corresp. $\delta$-Jacobi fields), on surfaces that are $\varepsilon$-close to being self-shrinkers (corresp. $\varepsilon$-geodesics) near a candidate for the self-shrinking torus.

The implications of such estimates are: The existence of a self-shrinking torus $\mathbb{T}^2$ (via existence of a smooth closed geodesic loop) with useful quantitative estimates of its geometry. Hence it gives conclusions about the Dirichlet and Neumann problems of the stability operator $\mathcal{L}$ on this "quantitative torus", leading to the main technical result below in Theorem \ref{TechThm} which is sufficient to prove Theorem \ref{Thm:torusphere}.

A thorough treatment of the background and details relating to this problem, and the construction as developed for general compact minimal hypersurfaces in $3$-manifolds by Nikolaos Kapouleas in \cite{Ka97}-\cite{Ka05} can be found in the recent \cite{Ka11}, the contents of which will not be described here (it should be noted that the highly symmetric special case enjoys significant simplifications compared to the full theory). Note that also the references \cite{Ng06}-\cite{Ng07}, and more recently \cite{KKM} and \cite{Ng11}, were concerned with gluing problems for (non-compact) self-shrinkers.

Self-shrinkers are minimal surfaces in Euclidean space with respect to a conformally changed Gaussian metric $g$:
\beq\label{GaussMetricToru}
\begin{gathered}
\Sigma^n\subseteq\Reals^{n+1}\quad\textrm{is a self-shrinker}\quad\Longleftrightarrow\quad H_{(\Reals^{n+1},g)}(\Sigma)=0,\\
g_{ij} = \frac{\delta_{ij}}{\exp\left(|X|^2/2n\right)},\quad X\in\Reals^{n+1}.
\end{gathered}
\eeq

Recall the constructions by Nicos Kapouleas (in \cite{Ka97}-\cite{Ka11}), concerning desingularization of a finite collection of compact minimal surfaces in a general ambient Riemannian $3$-manifold $(M^3,g)$. The conditions for the construction to work, in our situation, are the following, where the collection is identified with one immersed surface $\mathcal{W}$ with intersections along the (smooth) curve $\underline{\mathcal{C}}$, which can have several connected components.
\begin{conditions}[\cite{Ka05}-\cite{Ka11}]
\begin{itemize}
\item[]
\item[(I)] There are no points of triple intersection, all intersections are transverse and $\underline{\mathcal{C}}\cap\partial\mathcal{W}=\emptyset$ holds.
\item[(N1)] The kernel for the linearized operator
\beq
\mathcal{L} =\Delta+|A|^2+\Ric(\nu,\nu)\quad on \quad\mathcal{W},
\eeq
with Dirichlet conditions on $\partial\mathcal{W}$,
is trivial (unbalancing condition).

\item[(N2)] The kernel for the linearized operator
$\mathcal{L}$ on $\hat{\mathcal{W}}$,
with Dirichlet conditions on $\partial\hat{\mathcal{W}}$,
is trivial (flexibility condition).
\end{itemize}
\end{conditions}

Instead of (N1) one may substitute:
\begin{itemize}
\item[(N1')] \emph{The kernel for the linearized operator
$\mathcal{L}$ on $\mathcal{W}$,
with Neumann conditions on $\partial\mathcal{W}$,
is trivial (unbalancing condition).}
\end{itemize}

\begin{figure}
\centering
\includegraphics[width=\textwidth]{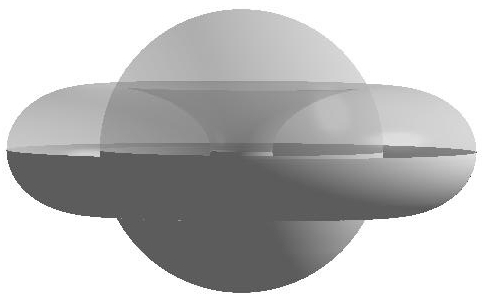}
\caption{The closed, embedded self-similar "toruspheres" $\Sigma_g^2\subseteq\Reals^3$ of genus $g$ in Theorem 1. \emph{Showing}: Immersed singular configuration, before handle insertion along two intersection curves (drawn with MATLAB).}
\end{figure}

The Neumann version (N1') of the non-degeneracy conditions will be used for the construction of the closed, embedded self-shrinkers where one solves the self-shrinker equations for graphs with the Neumann conditions over the circle of intersection of the torus and symmetry plane $\mathbb{T}^2\cap\mathcal{P}$, where $\mathbb{T}^2$ denotes a self-shrinking torus. Then the closed surfaces are obtained via simple doubling, by reflection through this plane, where standard regularity theory guarantees smoothness.

Yet another important version of the above conditions, is the one obtained by imposing symmetries, say under a large (discrete) subgroup $G\subseteq O(3)$, throughout the construction. Indeed, one may then restrict to verifying the non-degeneracy conditions (N1)--(N2) under the additional assumption of the symmetries in $G$.

While some properties of the Jacobi fields can be deduced from the known eigenvalues and eigenfunctions for the stability operator $\mathcal{L}$ (see e.g. \cite{CM2}), one does not seem to obtain enough accurate information for our purposes. We show how to estimate the quantities using the (generalized) Bellman-Gr\"{o}nwall's inequalities for second order Sturm-Liouville problems, and explicit test functions.

The care one needs to exercise when estimating the explicit constants, as well as the number and complexity of the barriers and test functions one needs to choose, becomes significant owing to mainly two factors: 1) The large Lipschitz constants of the PDE system (relative to the scale of the self-shrinking torus). The geometric reason that this happens is that the Gau\ss{} curvature in the metric on Angenent's upper half-plane has a maximal value of around 30 along the candidate torus (with the maximum occurring at the point nearest to the origin in $\Reals^3$), giving naive characteristic conjugate distances of down to $\frac{\pi}{\sqrt{K}}\sim 0.5$, while the circumference of the torus is around $\sim 7$, in the metric. Hence the solutions to Jacobi's equation can be expected to, and indeed do, oscillate several times around the circumference of $\mathbb{T}^2$, in a non-uniform way and yet just fall short of "matching up" smoothly. Furthermore: 2) The location of the conjugate points on $\mathbb{S}^2$ for the appropriate Jacobi equation is very near the singular curve $\underline{\mathcal{C}}$ (i.e. the boundaries of the connected components of $\mathcal{M}\setminus\underline{\mathcal{C}}$), which also requires tighter estimates.

Hence it takes work to strengthen the estimates to a useful form. One device to do this is what could be called the "sesqui"-shooting problem in Definition \ref{Sesqui}, for identifying the position of the torus, where "sesqui" refers to the fact that it is a double shooting problem but with a compatibility condition linking the two: That each pair of curves always meet at a simple, explicitly known solution. Here we take the round cylinder of radius $\sqrt{2}$ as reference. Since the errors are exponential in the integral of the Lipschitz constants, one may by virtue of the sesqui-shooting roughly take the square-root of the errors, which allows us to obtain bounds with explicit constants of the order of $10^1$--$10^2$ instead of $10^4$.

Although reduced in volume by the above discussion, it turns out that the explicit choice of test functions (f.ex. adequate piecewise polynomial choices can easily be found using a Taylor expansion fit at selected points, deduced from a numerical solution to the PDEs), to insert into the key estimates in Section \ref{sec:torus}, still take up several pages, and is not in itself very enlightening material to include. The computational job of checking the estimates for the many low-degree polynomials is in any case best carried out using computer software. The corresponding figures throughout this paper, showing Angenent's torus and the relevant Jacobi fields, were likewise computed and drawn in MATLAB.

Thus, the present paper does not contain the most efficient result that one could hope for, by amounting in effect to the reduction of the rigorous proof of the existence of closed self-shrinkers with genus to the still non-trivial algorithmic task of finding the zeros of several explicit low-order polynomials with integer coefficients (of course a task a computer easily manages rigorously).

\section{$\varepsilon$-Geodesics, $\delta$-Jacobi Fields: A Quantitative Self-Shrinking Torus $\mathbb{T}^2$} \label{sec:torus}
\subsection{Existence of $\mathbb{T}^2$, With Geometric Estimates}

In order to later understand precisely the kernel of the stability operators, we need to establish a detailed quantitative version of the existence of Angenent's torus. Towards the end of this section we will arrive at the basic, explicit estimates on the location and geometric quantities of such a torus, but we will first work out the explicit estimates for also the Jacobi equation before finalizing the choices in the estimate, in order to require as little as possible of the test functions.

\begin{figure}\label{TorusAndSphere}
\centering
\includegraphics[width=0.9\textwidth]{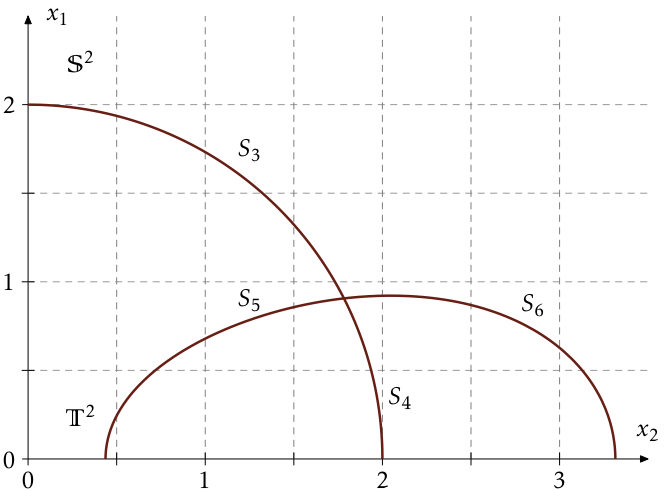}
\caption{The connected components of $\hat{\mathcal{W}}=(\mathbb{S}^2\cup\mathbb{T}^2)\setminus(\mathbb{S}^2\cap\mathbb{T}^2)$. The self-intersecting initial surface is generated by rotation around the $x_1$-axis and reflection in the $x_2$-axis.}
\end{figure}

We note that much less complicated estimates (and ditto test functions) would ensure the mere existence, leading to a different proof of the self-shrinking "doughnut" existence result from \cite{Ang}, in the $3$-dimensional case. But as explained, here we will need very precise estimates of several aspects of the geometry of $\mathbb{T}^2$.
\begin{proposition}
Let $\varepsilon_{\textrm{gap}}=10^{-3}$. There exists a closed, embedded, self-shrinking torus of revolution $\mathbb{T}^2$ with the following properties:
\begin{itemize}
\item[(1)] The torus is ($x_1\mapsto -x_1$)-symmetric.

\item[(2)] $\mathbb{T}^2$ intersects the $x_2$-axis orthogonally at two heights $a^+>a^->0$,
\begin{align}
\fracsm{4034}{1217}-\fracsm{5}{2}\varepsilon_{\textrm{gap}}&<a^+<\fracsm{4034}{1217}+\fracsm{5}{2}\varepsilon_{\textrm{gap}},\label{UpCross}\\
\fracsm{7}{16}-\fracsm{3}{98}&<a^-<\fracsm{7}{16}+\fracsm{3}{98}.\label{LowCross}
\end{align}

\item[(3)] $\mathbb{T}^2$ intersects the sphere $\mathbb{S}^2$ of radius 2 at two points $p^{\pm}=(\pm x_{\mathbb{S}^2},y_{\mathbb{S}^2})$, where:
\begin{align}
\left\|p^{\pm}-\left(\pm\fracsm{29}{32},\fracsm{41}{23}\right)\right\|_{\Reals^2}\leq 5\varepsilon_{\textrm{gap}}.\label{SphereCross}
\end{align}
Correspondingly, the angles $\angle(\pm\vec{e}_1,p^\pm)$ from the $x_1$-axis satisfy:
\beq\label{SphereCrossAngle}
\big|\angle(\pm\vec{e}_1,p^\pm)-\fracsm{11}{10}\big|\leq \fracsm{1}{20}.
\eeq
\end{itemize}

\end{proposition}
\begin{remark}
Note that there is no assertion or proof of uniqueness of $\mathbb{T}^2$ with these properties, although this is expected to be true.
\end{remark}

\begin{proof}

We first need to describe the double shooting problem (or rather "sesqui"-shooting problem, since we remove one parameter). For this we need an elementary lemma, which can either be proved (in a weaker version) using the approximation methods described later in this section, or by analysis directly of the ODE in (\ref{SSEqToru}).

\begin{lemma}
For each $d\in[0,\sqrt{2}]$ let $\gamma_d:[0,\infty)\to\Reals^+\times\Reals$ be the geodesic starting at $(0,d)$ with initial derivative $\gamma_d'(0)=(1,0)$ and consider the first time $t^0(d)$ such that $\gamma_d$ intersects the cylinder geodesic $\{x_2= \sqrt{2}\}$. Then the function $I:[0,\sqrt{2}]\to[0,\sqrt{2}]$ given by
\[
I(d):=x_1(\gamma_s(t^0(d))),
\]
is continuous and strictly increasing. It satisfies the bounds:
\[
1 \leq\sup_{d\in [\fracsm{13}{32},\fracsm{15}{32}]}I'(d)\leq\frac{5}{4}.
\]
\end{lemma}

\begin{definition}["Sesqui"-Shooting Problem]\label{Sesqui}
\begin{itemize}
\item[]
\item[(i)] Fix $a>\sqrt{2}$.
\item[(ii)] Let $\gamma_a:[0,t_a]\to\Reals^+\times\Reals$ be the fully extended geodesic contained in $\{x_1\geq 0\}$, with initial conditions $\gamma_a(0)=(0,a)$ and $\gamma_a'(0)=(1,0)$.
\item[(iii)] By the maximum principle for (\ref{SSEqToru}), we have $\gamma_a\cap\{x_2=\sqrt{2}\}\neq\emptyset$. Assume that the first point of $\gamma_a$ crossing $\sqrt{2}$ belongs to $\{0<x_1\leq \sqrt{2}\}$, and denote the time it happens by $t^0(a)$.
\item[(iv)]
By the preceding Lemma, there exists a unique $b(a)\in [0,\sqrt{2}]$ and $s^0(a)$ such that for $\gamma_b$ on $[0,s^0]$ with $\gamma_b(0)=(0,b)$ and $\gamma_b'(0)=(1,0)$, we have
\[
\gamma_a(t^0(a))=\gamma_b(s^0(a)).
\]
\item[(iv)] Define $\Phi(a)$ by
\[
\Phi(a):=\frac{\vec{e_3}\cdot(\gamma'_a(t^0(a))\times\gamma'_b(s^0(a))}{|\gamma'_a(t^0(a))||\gamma'_b(s^0(a))|},
\]
or equivalently the oriented angle between the tangents of $\gamma_a$ and $\gamma_b$ at the intersection point.
\end{itemize}
\end{definition}
The following lemma is clear from the definition:

\begin{lemma}
The function $\Phi$ in Definition \ref{Sesqui} is well-defined and continuous, on the open connected set of values $a>\sqrt{2}$ with the property that the first point of intersection of $\gamma_a$ with $\{x_2=\sqrt{2}\}$ belongs to $\{x_1\in(0,\sqrt{2})\}$.
\end{lemma}

The strategy for the proof of the proposition will now be to prove for two different nearby pairs $(a^+,b^+)$ and $(a^-,b^-)$, that will be chosen such that they are geodesics broken at the cylinder $\{x_2=\sqrt{2}\}$, that:
\begin{align}
&\Phi(a^+)>0,\label{BigPhi}\\
&\Phi(a^-)<0.\label{SmallPhi}
\end{align}
Existence then follows, from the intermediate value theorem, of a pair $(a^0,b^0)$ such that:
\begin{align}
&\Phi(a^0)=0,\\
&a^-<a^0<a^+,\\
&b^-<b^0<b^+.
\end{align}
The estimates leading to the proof of (\ref{BigPhi})--(\ref{SmallPhi}) will then lead to the estimates in the proposition.

Consider now a curve (parametrized by any parameter $t$),
\[
\gamma(t)=(x(t),y(t)).
\]

Then the equation for $\gamma$ to generate a self-shrinker by rotation reads:
\[
x''y'-y''x'=\left[\frac{yx'-xy'}{2}-\frac{x'}{y}\right]\left((x')^2+(y')^2\right).
\]

We need the following simple lemma:
\begin{lemma}
If we define the operator $\mathcal{M}_1$ acting on $C^2$-functions $u:I\to\Reals$, for some interval $I\subseteq[0,\infty]$, by
\beq
\mathcal{M}_1(u,p):=\Big[\frac{x_1p-u}{2}+\frac{1}{u}\Big]\big(1+p^2\big),
\eeq
then the function $u(x_1)$, a graph over the $x_1$-axis, generates a self-shrinker by rotation (around $x_1$) if and only if
\beq
u''=\mathcal{M}_1(u,u').
\eeq
Likewise for graphs over the $x_2$-axis, defining
\beq
\mathcal{M}_2(f,q):= \left[ \left(\frac{x_2}{2} - \frac{1}{x_2}\right) q - \frac{f}{2} \right] \left(1 + q^2 \right),
\eeq
again characterizes such solutions via
\beq
f''=\mathcal{M}_2(f,f').
\eeq
\end{lemma}

The lemma is shown directly from the definition of a self-shrinker (see \cite{KM}). We now compute that:
\begin{align}
&\frac{\partial}{\partial u}\mathcal{M}_1(u,p)=\left[-\frac{1}{2}-\frac{1}{u^2}\right](1+p^2),\\
&\frac{\partial}{\partial p}\mathcal{M}_1(u,p)=\frac{x_1}{2}(1+p^2)+2p\left[\frac{x_1 p -u}{2}+\frac{1}{u}\right],\\
&\frac{\partial}{\partial f}\mathcal{M}_2(f,q)=-\frac{1}{2}\left(1 + q^2 \right),\\
&\frac{\partial}{\partial q}\mathcal{M}_2(f,q)=\left(\frac{x_2}{2} - \frac{1}{x_2}\right)(1+3q^2)-fq.
\end{align}

We will first consider graphs $u:[0,\fracsm{3}{5}]\to\Reals$. Now, considering an approximate solution $U$, fix the quantity $\varepsilon_1^T$, later to be chosen, which reflects the order of magnitude of the precision with which we wish to determine the position of $\mathbb{T}^2$ on this interval.

\beq
\mathcal{M}_1 (U,U')-\mathcal{M}_1(u,u')=(U-u)\frac{\partial}{\partial u}\mathcal{M}_1(\xi,u')+(U'-u')\frac{\partial}{\partial p}\mathcal{M}_1(U,\xi'),
\eeq
for some $\xi\in[U(x_1),u(x_1)]$ and $\xi'\in[U'(x_1),u'(x_1)]$.

\underline{Top of the torus: $x_1$-graph}

Assume we have the uniform estimates:
\[
|U''-\mathcal{M}_1(U,U')|\leq\varepsilon_1^T,
\]
for some $\varepsilon_1^T$. Assume also the following bounds:
\begin{align*}
&\int_0^{x_1}\left|\frac{1}{2}+\frac{1}{\xi^2}\right|(1+(u'(x_1))^2)dx_1\leq \frac{4}{5}x_1,\quad x_1\in[0,\fracsm{3}{5}],\quad |\xi-U(x_1)|\leq\varepsilon_0,\\
&\int_0^{x_1}\left|\frac{x_1}{2}(1+(\xi')^2)+2\xi'\left[\frac{x_1 \xi' -U(x_1)}{2}+\frac{1}{U(x_1)}\right]\right|dx_1\leq \frac{8}{3}x_1^2,\quad x_1\in[0,\fracsm{3}{5}],\quad |\xi'-U'(x_1)|\leq\varepsilon_0.
\end{align*}
If we let $\varphi(x_1):=|U'(x_1)-u'(x_1)|$, we can now estimate (recall that one has $||f'|'|=|f''|$ almost everywhere):
\beq
\begin{split}
\varphi'(x_1)&\leq \left||U'-u'|'\right|\overset{\textrm{a.e}}{=}|U''-u''|\\
&\leq |\mathcal{M}_1 (U,U')-\mathcal{M}_1(u,u')|+\varepsilon_1^T\\
&\leq |\partial_u\mathcal{M}_1(\xi,u')||U-u|+|\partial_p\mathcal{M}_1(U,\xi')||U'-u'|+\varepsilon_1^T\\
&\leq |\partial_p\mathcal{M}_1(U,\xi')|\varphi(x_1) + |\partial_u\mathcal{M}_1(\xi,u')|\int_{0}^{x_1}\varphi(s)ds+|\partial_u\mathcal{M}_1(\xi,u')||U(0)-u(0)|+\varepsilon_1^T.
\end{split}
\eeq
Thus we integrate this inequality, which holds almost everywhere with respect to the Lebesgue measure, and get:
\begin{align*}
\varphi(x_1)&\leq \varphi(0)+\int_{0}^{x_1}|\partial_p\mathcal{M}_1(U,\xi')(s)|\varphi(s)ds +\int_{0}^{x_1}|\partial_u\mathcal{M}_1(\xi,u')(t)|\int_{0}^{t}\varphi(s)dsdt\\
&\quad\quad\quad+ \fracsm{4}{5}x_1|U(0)-u(0)|+\varepsilon_1^Tx_1\\
& \leq \varphi(0) + \fracsm{4}{5}x_1|U(0)-u(0)|+\varepsilon_1^Tx_1+\int_{0}^{x_1}\Big[|\partial_p\mathcal{M}_1(U,\xi')(s)|+\fracsm{4}{5}x_1\Big]\varphi(s)ds.
\end{align*}
We are now ready to use the integral form of Gr\"o{}nwall-Bellman's inequality, namely for $\alpha(t)$ a non-decreasing function:
\[
\forall t\in I: \varphi(t)\leq \alpha(t)+\int_{a}^t \beta(s)\varphi(s)ds\quad\Longrightarrow\quad\forall t\in I:  \varphi(t)\leq \alpha(t)\exp\left\{\int_{a}^t\beta(s)ds\right\}.
\]
Here we thus conclude from the above, that
\beq
\varphi(x_1)\leq \Big[\varphi(0) + \fracsm{4}{5}x_1|U(0)-u(0)|+\varepsilon_1^Tx_1\Big]\exp\Big\{\fracsm{8}{3}x_1^2+\fracsm{4}{5}x_1^2\Big\},
\eeq
so that here we obtain the estimates (for $\varphi(0)=0$):
\begin{align*}
|U'(x_1)-u'(x_1)|&\leq \fracsm{3}{5}\exp(\fracsm{156}{125})(\varepsilon_1^T+\fracsm{4}{5}|U(0)-u(0)|),\\
|U(x_1)-u(x_1)|&\leq |U(0)-u(0)|+(\varepsilon_1^T+\fracsm{4}{5}|U(0)-u(0)|)\int_0^{\fracsm{3}{5}}s\exp\big(\fracsm{52}{15}s^2\big)ds\\
&=|U(0)-u(0)|+\fracsm{15}{104}\big(\exp\big(\fracsm{52}{15}(\fracsm{3}{5})^2\big)-1\big)(\varepsilon_1^T+\fracsm{4}{5}|U(0)-u(0)|)\\
&\leq \fracsm{23}{80}|U(0)-u(0)|+\fracsm{9}{25}\varepsilon_1^T.
\end{align*}

\underline{Top of the torus: $x_2$-graph to the sphere}

We continue with the next part, which is graphical over the $x_2$-axis. In this region we again let $\psi(x_2):=|F'-f'|$. Here we will assume the estimate
\beq
|F''-\mathcal{M}_2(F,F')|\leq\varepsilon_2^T,
\eeq
and furthermore the estimates on $x_2\in[y_{\mathbb{S}^2},u_T(3/5)]$
\begin{align*}
&I^{(2)}_u=\int_{x_2}^{u_T(3/5)}\left|\frac{1}{2}(1+(F')^2)\right|dx_2\leq\fracsm{13}{8}-\fracsm{x_2}{2},\\
&I^{(2)}_p=\int_{x_2}^{u_T(3/5)}\left|\left(\frac{x_2}{2} - \frac{1}{x_2}\right)(1+3(\xi')^2)-F(x_2) \xi'\right|dx_2\leq \fracsm{7}{4}-\fracsm{9}{10}\left(x_2-y_{\mathbb{S}^2}\right)^2.
\end{align*}
We can then estimate as follows:
\begin{align*}
\psi'(x_2)&\leq \left|\partial_u\mathcal{M}_2\right||F-f|+\left|\partial_p\mathcal{M}_2\right||F'-f'|+\varepsilon_2^T\\
&\leq \left|\partial_u\mathcal{M}_2\right|\int_{x_2}^{a}\psi(s)ds+\left|\partial_p\mathcal{M}_2\right|\psi(x_2)+\left|\partial_u\mathcal{M}_2\right||F(a)-f(a)|+\varepsilon_2^T,
\end{align*}
which integrates to
\begin{align}
\psi(x_2)\leq &\int_{x_2}^{u_T(3/5)}\left[|\partial_p\mathcal{M}_2|+\fracsm{13}{8}-\fracsm{x_2}{2}\right]\psi(s)ds+(\fracsm{13}{8}-\fracsm{x_2}{2})|F({u_T(3/5)})-f({u_T(3/5)})|\\
&\qquad+\varepsilon_2^T(u_T(3/5)-x_2)+\psi(a),
\end{align}
so that, again by Gr\"o{}nwall-Bellman,
\begin{align*}
\psi(x_2)\leq &\Big[(\fracsm{13}{8}-\fracsm{x_2}{2})|F(a)-f(a)|+|F'(a)-f'(a)|+\varepsilon_2^T(u_T(3/5)-x_2)\Big]\times\\
&\exp\left\{\fracsm{7}{4}-\fracsm{9}{10}\left(x_2-y_{\mathbb{S}^2}\right)^2+\left(\fracsm{13}{8}-\fracsm{x_2}{2}\right)\left(u_T(3/5)-x_2\right)\right\}.
\end{align*}
The endpoint estimates for $|F'(x_2) - f'(x_2)|$ are thus:
\begin{align*}
\psi(y_{\mathbb{S}^2})
&\leq \fracsm{59}{4}|F(u_T(3/5))-f(u_T(3/5))|+\fracsm{54}{5}|F'(u_T(3/5))-f'(u_T(3/5))|+\fracsm{377}{20}\varepsilon_2^T\\
&\leq \fracsm{59}{4}\frac{|U(F(u_T(3/5)))-u(f(u_T(3/5)))|}{|U'(F(\xi))|}+\fracsm{54}{5}\frac{|U'(F(u_T(3/5)))-u'(f(u_T(3/5)))|}{|U'(F(u_T(3/5)))||u'(f(u_T(3/5)))|}+\fracsm{377}{20}\varepsilon_2^T\\
&\leq \fracsm{59}{4}\fracsm{10}{11}\left(\fracsm{9}{25}\varepsilon_1^T+\fracsm{23}{80}|U(0)-u(0)|\right)+\fracsm{54}{5}\left(\fracsm{10}{11}\right)^2\fracsm{3}{5}\exp(\fracsm{156}{125})(\varepsilon_1^T+\fracsm{4}{5}|U(0)-u(0)|)+\fracsm{377}{20}\varepsilon_2^T\\
&\leq 27|U(0)-u(0)|+34\varepsilon_1^T+19\varepsilon_2^T.
\end{align*}
Integrating the above estimate for $\psi(x_2)$, we also get:
\begin{align*}
|F(y_{\mathbb{S}^2})-f(y_{\mathbb{S}^2})|\leq & (1+\fracsm{63}{8})|F(u_T(3/5))-f(u_T(3/5))|+\fracsm{83}{20}|F'(u_T(3/5))-f'(u_T(3/5))|+\fracsm{137}{20}\varepsilon_2^T\\
\leq &12|U(0)-u(0)|+15\varepsilon_1^T+7\varepsilon_2^T.
\end{align*}

\underline{Top of the torus, $x_2$-graph from sphere to cylinder}

We consider the region between the sphere and cylinder, and let again $\psi(x_2):=|F'-f'|$. Here we will assume the estimate
\beq
|F''-\mathcal{M}_2(F,F')|\leq\varepsilon_3^T,
\eeq
and furthermore the estimates on $x_2\in[y_{\mathbb{S}^2},\sqrt{2}]$
\begin{align*}
&I^{(2)}_u=\int_{x_2}^{y_{\mathbb{S}^2}}\left|\frac{1}{2}(1+(F')^2)\right|dx_2\leq\fracsm{1}{4}(y_{\mathbb{S}^2}-x_2),\\
&I^{(2)}_p=\int_{x_2}^{y_{\mathbb{S}^2}}\left|\left(\frac{x_2}{2} - \frac{1}{x_2}\right)(1+3(\xi')^2)-F(x_2) \xi'\right|dx_2\leq \fracsm{27}{50}(y_{\mathbb{S}^2}-x_2).
\end{align*}

We estimate as follows:
\begin{align*}
\psi'(x_2)&\leq \left|\partial_u\mathcal{M}_2\right||F-f|+\left|\partial_p\mathcal{M}_2\right||F'-f'|+\varepsilon_2^T\\
&\leq \left|\partial_u\mathcal{M}_2\right|\int_{x_2}^{y_{\mathbb{S}^2}}\psi(s)ds+\left|\partial_p\mathcal{M}_2\right|\psi(x_2)+\left|\partial_u\mathcal{M}_2\right||F(y_{\mathbb{S}^2})-f(y_{\mathbb{S}^2})|+\varepsilon_3^T,
\end{align*}
which integrates to:
\begin{align}
\psi(x_2)\leq & \psi(y_{\mathbb{S}^2})+\int_{x_2}^{y_{\mathbb{S}^2}}\left[|\partial_p\mathcal{M}_2|+\fracsm{1}{4}(y_{\mathbb{S}^2}-x_2)\right]\psi(s)ds+\fracsm{1}{4}(y_{\mathbb{S}^2}-x_2)|F(y_{\mathbb{S}^2})-f(y_{\mathbb{S}^2})|+\varepsilon_3^T(y_{\mathbb{S}^2}-x_2),
\end{align}
so that, again by Gr\"o{}nwall-Bellman,
\begin{align*}
\psi(x_2)\leq &\Big[\fracsm{1}{4}(y_{\mathbb{S}^2}-x_2)|F(y_{\mathbb{S}^2})-f(y_{\mathbb{S}^2})|+|F'(y_{\mathbb{S}^2})-f'(y_{\mathbb{S}^2})|+\varepsilon_3^T(y_{\mathbb{S}^2}-x_2)\Big]\times\\
&\exp\left\{\fracsm{27}{50}(y_{\mathbb{S}^2}-x_2)+\fracsm{1}{4}(y_{\mathbb{S}^2}-x_2)^2\right\}.
\end{align*}

Inserting the previous estimate gives:
\begin{align*}
\psi(\sqrt{2})&\leq \fracsm{1}{8}|F(y_{\mathbb{S}^2})-f(y_{\mathbb{S}^2})|+\fracsm{4}{3}|F'(y_{\mathbb{S}^2})-f'(y_{\mathbb{S}^2})|+\fracsm{1}{2}\varepsilon_3^T\\
&\leq \fracsm{1}{8}\left(12|U(0)-u(0)|+15\varepsilon_1^T+7\varepsilon_2^T\right)+\fracsm{4}{3}\left(27|U(0)-u(0)|+34\varepsilon_1^T+19\varepsilon_2^T\right)+\fracsm{1}{2}\varepsilon_3^T.
\end{align*}
Hence, we finally obtain the estimate:
\begin{align*}
|F'(\sqrt{2})-f'(\sqrt{2})|&\leq 36|U(0)-u(0)|+45\varepsilon_1^T+25\varepsilon_2^T+\fracsm{1}{2}\varepsilon_3^T.
\end{align*}
Integrating the estimates, we also get:
\begin{align*}
|F(y_{\mathbb{S}^2})-f(y_{\mathbb{S}^2})|\leq & (1+\fracsm{1}{50})|F(y_{\mathbb{S}^2})-f(y_{\mathbb{S}^2})|+\fracsm{33}{80}|F'(y_{\mathbb{S}^2})-f'(y_{\mathbb{S}^2})|+\fracsm{2}{25}\varepsilon_3^T\\
\leq &24|U(0)-u(0)|+30\varepsilon_1^T+15\varepsilon_2^T+\fracsm{2}{25}\varepsilon_3^T.
\end{align*}

\underline{Bottom of the torus: $x_1$-graph from the plane}

Assume uniform estimates:
\[
|U''-\mathcal{M}_1(U,U')|\leq\varepsilon^B_1,
\]
for some $\varepsilon_1^B$. Assume also, for the argument, the following bounds (for  $x_1\in[0,\fracsm{1}{2}]$):
\begin{align*}
&\int_0^{x_1}\left|\frac{1}{2}+\frac{1}{\xi^2}\right|(1+(u'(x_1))^2)dx_1\leq \frac{29}{5}x_1+\frac{1}{20},\quad |\xi-U(x_1)|\leq\varepsilon_0,\\
&\int_0^{x_1}\left|\frac{x_1}{2}(1+(\xi')^2)+2\xi'\left[\frac{x_1 \xi' -U(x_1)}{2}+\frac{1}{U(x_1)}\right]\right|dx_1\leq 4x_1^2+\frac{1}{5}x_1,\quad |\xi'-U'(x_1)|\leq\varepsilon_0.
\end{align*}
As always, we get with $\varphi(x)=|U'(x)-u'(x)|$:
\begin{align*}
\varphi(x_1)&\leq \varphi(0)+\int_{0}^{x_1}|\partial_p\mathcal{M}_1(U,\xi')(s)|\varphi(s)ds +\int_{0}^{x_1}|\partial_u\mathcal{M}_1(\xi,u')(t)|\int_{0}^{t}\varphi(s)dsdt\\
&\quad\quad\quad+ (\fracsm{29}{5}x_1+\fracsm{1}{20})|U(0)-u(0)|+\varepsilon_1^Bx_1\\
& \leq \varphi(0) + (\fracsm{29}{5}x_1+\fracsm{1}{20})|U(0)-u(0)|+\varepsilon_1^Bx_1+\int_{0}^{x_1}\Big[|\partial_p\mathcal{M}_1(U,\xi')(s)|+\fracsm{29}{5}x_1+\fracsm{1}{20}\Big]\varphi(s)ds.
\end{align*}
Using once again Gr\"o{}nwall-Bellman's inequality,
\beq
\varphi(x_1)\leq \Big[\varphi(0) + (\fracsm{29}{5}x_1+\fracsm{1}{20})|U(0)-u(0)|+\varepsilon_1^Bx_1\Big]\exp\Big\{\fracsm{29}{5}x^2_1+\fracsm{x_1}{20}+4x_1^2+\fracsm{1}{5}x_1\Big\},
\eeq
and we obtain the estimates (for $\varphi(0)=0$):
\begin{align*}
|U'(x_1)-u'(x_1)|&\leq \Big[(\fracsm{29}{5}x_1+\fracsm{1}{20})|U(0)-u(0)|+\varepsilon_1^Bx_1\Big]\exp\Big\{\fracsm{49}{5}x^2_1+\fracsm{x_1}{4}\Big\}\\
|U'(\fracsm{1}{2})-u'(\fracsm{1}{2})|&\leq 39|U(0)-u(0)|+\fracsm{28}{5}\varepsilon^B_1,\\
|U(x_1)-u(x_1)|&\leq |U(0)-u(0)|\left(1+\int_0^{x_1}(\fracsm{29}{5}s+\fracsm{1}{20})\exp\Big\{\fracsm{49}{5}s^2+\fracsm{s}{4}\Big\}ds\right)\\
&+\varepsilon^B_1\int_0^{x_1}s\exp\Big\{\fracsm{49}{5}s^2+\fracsm{s}{4}\Big\}ds\\
|U(\fracsm{1}{2})-u(\fracsm{1}{2})|&\leq \fracsm{23}{5}|U(0)-u(0)|+\fracsm{3}{5}\varepsilon^B_1.
\end{align*}

\underline{Bottom: $x_2$-graph to cylinder}

We continue with the next part, which is graphical over the $x_2$-axis. In this region we again let $\psi(x_2)=\psi_B(x_2):=|F'-f'|$. Here we will assume the estimate
\beq
|F''-\mathcal{M}_2(F,F')|\leq\varepsilon^B_2,
\eeq
and furthermore the estimates on $x_2\in[a_0,\sqrt{2}]$ (where $|a_0-\fracsm{28}{39}|\leq \varepsilon_0$).
\begin{align*}
&I^{(2)}_u=\int_{a_0}^{x_2}\left|\frac{1}{2}(1+(F')^2)\right|dx_2\leq\frac{16}{25}x_2-\frac{3}{7},\\
&I^{(2)}_p=\int_{a_0}^{x_2}\left|\left(\frac{x_2}{2} - \frac{1}{x_2}\right)(1+3(\xi')^2)-F(x_2) \xi'\right|dx_2\leq \frac{9}{10}-\frac{11}{10}\left(x_2-\frac{3}{2}\right)^2.
\end{align*}
We can then estimate as follows:
\begin{align*}
\psi'(x_2)&\leq \left|\partial_u\mathcal{M}_2\right||F-f|+\left|\partial_p\mathcal{M}_2\right||F'-f'|+\varepsilon_2^B\\
&\leq \left|\partial_u\mathcal{M}_2\right|\int_a^{x_2}\psi(s)ds+\left|\partial_p\mathcal{M}_2\right|\psi(x_2)+\left|\partial_u\mathcal{M}_2\right||F(a)-f(a)|+\varepsilon_2^B,
\end{align*}
which integrates to
\begin{align}
\psi(x_2)\leq &\int_{a_0}^{x_2}\left[|\partial_p\mathcal{M}_2|+\fracsm{16}{25}x_2-\fracsm{3}{7}\right]\psi(s)ds+(\fracsm{16}{25}x_2-\fracsm{3}{7})|F({a_0})-f({a_0})|+\varepsilon_2^B(x_2-a_0)+\psi(a).
\end{align}
By Gr\"o{}nwall-Bellman,
\begin{align*}
\psi(x_2)\leq &\Big[(\fracsm{16}{25}x_2-\fracsm{3}{7})|F(a)-f(a)|+|F'(a)-f'(a)|+\varepsilon_2^B(x_2-a_0)\Big]\times\\
&\exp\left\{\fracsm{9}{10}-\fracsm{11}{10}\left(x_2-\fracsm{3}{2}\right)^2+\left(x_2-a_0\right)\left(\fracsm{16}{25}x_2-\fracsm{3}{7}\right)\right\}\\
&\leq \fracsm{13}{8}|F(a_0)-f(a_0)|+\fracsm{17}{5}|F'(a_0)-f'(a_0)|+\fracsm{69}{29}\varepsilon^B_2\\
&\leq \fracsm{13}{8}\frac{|U(F(a_0))-u(f(a_0))|}{|U'(F(\xi))|}+\fracsm{17}{5}\frac{|U'(F(a_0))-u'(f(a_0))|}{|U'(F(a_0))||u'(f(a_0))|}+\fracsm{69}{29}\varepsilon^B_2\\
&\leq \fracsm{13}{8}\fracsm{10}{12}\left(\fracsm{3}{5}\varepsilon^B_1+\fracsm{23}{5}|U(0)-u(0)|\right)+\fracsm{17}{5}\left(\fracsm{10}{12}\right)^2(\fracsm{28}{5}\varepsilon_1^B+39|U(0)-u(0)|)+\fracsm{69}{29}\varepsilon^B_2\\
&\leq 100|U(0)-u(0)|+15\varepsilon^B_1+\fracsm{5}{2}\varepsilon^B_2.
\end{align*}
Integrating the first line of this estimate, we also get:
\begin{align*}
|F(\sqrt{2})-f(\sqrt{2})|\leq & (1+\fracsm{23}{50})|F(a_0)-f(a_0)|+\fracsm{14}{9}|F'(a_0)-f'(a_0)|+\fracsm{2}{3}\varepsilon_2\\ 
\leq &48|U(0)-u(0)|+7\varepsilon^B_1+\fracsm{2}{3}\varepsilon^B_2.
\end{align*}

Hence one arranges that for the test functions $\gamma_\textrm{up}, \gamma_\textrm{low}:[0,1]\to\Reals\times\Reals_+$:
\begin{align}
\gamma_\textrm{down}(0)=(0,\fracsm{4034}{1217}-\fracsm{5}{2}\varepsilon_{\textrm{ud}})\\
\gamma_\textrm{up}(0)=(0,\fracsm{4034}{1217}+\fracsm{5}{2}\varepsilon_{\textrm{ud}}).
\end{align}
Then, as explained earlier, the intermediate value theorem applied to the sesqui-shooting problem implies the existence of the torus with the estimates (\ref{UpCross})--(\ref{LowCross}).

To see Property (3), we recall that
\beq
|F(y_{\mathbb{S}^2})-f(y_{\mathbb{S}^2})|\leq \fracsm{317}{25}|U(0)-u(0)|+11\varepsilon_1+\fracsm{27}{5}\varepsilon_2,
\eeq
from which the estimate (\ref{SphereCross}) follows.

\end{proof}

\subsection{The Stability Operator $\mathcal{L}$ on Subsets of $\mathbb{S}^2$ and $\mathbb{T}^2$}

Recall the self-shrinker equation for a smooth oriented surface $S\subseteq\Reals^3$ to be a self-shrinker (shrinking towards the origin with singular time $T=1$) is
\beq\label{SSEqToru}
H_S(\vec{X})- \fracsm{1}{2} \vec{X} \cdot \vec{\nu}_S(\vec{X}) = 0,
\eeq
for each $\vec{X} \in S$, where by convention $H_S=\sum_1^n\kappa_i$ is the sum of the signed principal curvatures w.r.t. the chosen normal $\vec{\nu}_S$ (i.e. $H=2$ for the sphere with outward pointing $\vec{\nu}$). We have normalized Equation (\ref{SSEqToru}) so that $T =  1$ is the singular time. 

For a smooth  normal  variation $\vec{X}_t$ determined by a function $u$ via $X_t=X_0+tu\vec{\nu}_{S}$, where $\vec{X}_0$ parametrizes $S$, the pointwise linear change in (minus) the quantity on the left hand side in (\ref{SSEqToru}) is given by the stability operator (see the Appendix, and also \cite{CM6}-\cite{CM7} for more properties of this operator)
\begin{equation} \label{SSEqToru_linearization}
\mathcal{L}_{S} u = \Delta_{S}  u+ |A_{S}|^2 u - \fracsm{1}{2} \left(\vec{X}\cdot\nabla_{S} u - u  \right).
\end{equation}

We are now ready to prove the main technical theorem in this paper, concerning the kernel on the connected components (with boundaries) of
\[
(\mathbb{S}^2\cup\mathbb{T}^2)\setminus(\mathbb{S}^2\cap\mathbb{T}^2),
\]
where $\mathbb{T}^2$ is the accurately estimated "quantitative" torus.

\begin{theorem}\label{TechThm}
There exists $N>0$ large enough that for the below six surfaces with boundary $S_1,\ldots,\mathcal{S}_6$,
\[
\ker\mathcal{L}_{S_i}=\{0\}\quad\textrm{for}\quad i=1,\ldots,6\quad\textrm{[w/ indicated boundary conditions]},
\]
when imposing at least $N$-fold rotational symmetry:
\begin{itemize}
\item[(1)] The surfaces $S_1=\mathbb{S}^2\cap\{x_1\geq 0\}$ and $S_2=\mathbb{T}^2\cap\{x_1\geq 0\}$ [Neumann conditions].
\item[(2)] The components $S_3$ and $S_4$ of $\{x_1\geq 0\}\cap\mathbb{S}^2\setminus(\mathbb{S}^2\cap\mathbb{T}^2)$ [Neumann conditions on $\{x_1=0\}$, Dirichlet conditions elsewhere].
\item[(3)] The components $S_5$ and $S_6$ of $\{x_1\geq 0\}\cap\mathbb{T}^2\setminus(\mathbb{S}^2\cap\mathbb{T}^2)$ [Neumann conditions on $\{x_1=0\}$, Dirichlet conditions elsewhere].
\end{itemize}
\end{theorem}

Recall first the variational characterization of shrinkers, as critical points of the functional
\beq
A(\Sigma)=\int_{\Sigma} e^{-|X(p)|^2/4}dA(p).
\eeq

As exploited in \cite{KKM} in the planar case, the conjugation identity connecting the stability operator in the Gaussian density to the linearized operator in (\ref{SSEqToru_linearization}) can be useful for the analysis. The general identity on any surface in $\Reals^3$ is stated in the following lemma.

\begin{lemma}\label{Conjug}
Let $\Sigma^2\subseteq\Reals^3$ be a smooth self-shrinker. Then the following conjugation identity holds for the stability operator on $\Sigma$:
\beq
\mathcal{L}=e^{\frac{|X|^2}{8}}\Big(\Delta +|A|^2-\frac{|X|^2+|X^\perp|^2}{16}+1\Big)e^{\frac{-|X|^2}{8}}.
\eeq
\end{lemma}
\begin{proof}[Proof of Lemma \ref{Conjug}]
By an elementary computation using Appendix C in \cite{KKM} and noting that $|X^\perp|=2|H|$.
\end{proof}

We let $\omega=(x')^2+(y')^2$ and express functions on the surface of rotation in coordinates as $v=v(t,\theta)$ and get for the intrinsic Laplacian that
\[
\Delta = \frac{1}{y\sqrt{\omega}}\frac{\partial}{\partial t}\Big(\frac{y}{\sqrt{\omega}}\frac{\partial}{\partial t}\Big)+\frac{1}{y^2}\frac{\partial^2}{\partial\theta^2}.
\]
The square of the second fundamental form is
\begin{align}
|A|^2&=\frac{1}{\omega^3}\left[(x''y'-y''x')^2+\frac{(x')^2\omega^2}{y^2}\right]\\
& = \frac{1}{\omega}\left[\left(\frac{yx'-xy'}{2}-\frac{x'}{y}\right)^2+\frac{(x')^2}{y^2}\right],
\end{align}
Recall that
\[
K_{e^{2w}g_0}=e^{-2\omega}\left(-\Delta_{g_0}\omega+K_{g_0}\right),
\]

so that the Gau\ss{} curvature of Angenent's metric is
\[
K_{\textrm{Ang}} = \frac{e^{\fracsm{x^2+y^2}{2}}}{y^2}\left(1+\frac{1}{y^2}\right)
\]

The remaining terms in the expression for the stability operator give:
\[
-\fracsm{1}{2}X\cdot \grad v +\fracsm{1}{2}v= -\fracsm{1}{2}\frac{xx'+yy'}{\omega}\frac{\partial v}{\partial t}+\fracsm{1}{2}v.
\]

By virtue of rotational symmetry, the equation $\mathcal{L}v=0$ separates, and we expand $v$ by its Fourier series on each radial circle:
\beq\label{FourierModes}
v(t,\theta)=\sum_{m\in\mathbb{Z}} v_m(t)e^{im\theta},
\eeq
and thus the equations we study are
\beq
\mathcal{L}_mv_m=0,
\eeq
for the appropriate boundary conditions (e.g. Dirichlet or Neumann), where
\beq
\begin{split}
\mathcal{L}_mv_m=&\frac{1}{y\sqrt{\omega}}\frac{\partial}{\partial t}\Big(\frac{y}{\sqrt{\omega}}\frac{\partial}{\partial t}\Big)v_m-\fracsm{1}{2}\frac{xx'+yy'}{\omega}v_m'+\Big(|A|^2+\fracsm{1}{2}-\frac{m^2}{y^2}\Big)v_m.
\end{split}
\eeq

Observing how the size of the 0th order coefficient improves with increasing symmetry, we thus immediately conclude the following proposition:
\begin{proposition}\label{MaxPrinc}
On the compact surface of revolution $S_\gamma$ with boundary, generated by the curve $\gamma$, we let:
\beq
M_0(\gamma):=\sup_{p\in\gamma}\Big[y(p)\sqrt{\fracsm{1}{2}+|A(p)|^2}\Big].
\eeq
Then the unique solutions to each of the above Sturm-Liouville problems (with the above-mentioned appropriate boundary condition for each case) are:
\beq
v_m\equiv0,\quad\textrm{for all}\quad m\geq M_0(\gamma).
\eeq
\end{proposition}
\begin{proof}
This follows immediately from an application of the usual maximum principle in one variable.
\end{proof}

The usage of the preceding proposition is now via the assumption of $k$-dihedral symmetry, which is built into the entire construction. It then follows, that only such Fourier modes $m$ in the separation Equation (\ref{FourierModes}) satisfying $k\mid m$ may appear in the decomposition. Hence by restricting to large enough values of $k$,
\beq
k\geq \max_i\:M_0(\gamma_i),
\eeq
and correspondingly large genus $g$, then by virtue of Proposition \ref{MaxPrinc}, there can be assumed to be no $m>0$ modes in the decomposition. Note of course, that since $m=0$ modes are $k$-symmetric for any $k$ (i.e. amounting to the tautology $k\mid 0)$, the argument does not apply to $m=0$, and presence of this mode must be considered separately.

We now therefore focus solely on the $m=0$ mode. Recalling the second variation formula for the energy of a curve, we get:
\begin{align}
\frac{d^2}{ds^2}A(\Sigma)&=\frac{d^2}{ds^2}\int_\Sigma e^{-\frac{|X|^2}{4}}dA_{\Reals^3}\\
&=-\int (f\mathcal{L} f)e^{\frac{|X|^2}{4}}dA_{\Reals^3}\\
&=\frac{d^2}{dv^2}\left(\mathrm{vol}(\mathbb{S}^{n-1})\mathrm{Length}_{\mathrm{Ang}}\right)\\
&=\mathrm{vol}(\mathbb{S}^{n-1})\int \left\langle-\frac{DV}{dt}-R(\gamma',V)\gamma',V\right\rangle dt\\
&=-\mathrm{vol}(\mathbb{S}^{n-1})\int_a^b g(g''+K_{\mathrm{Ang}}g)dt,
\end{align}
where $dA_{\Reals^3}$ means the surface measure induced by the standard Euclidean metric.

Note that we can therefore rewrite the $m=0$ equation in terms of $\tilde{v}_0=\sqrt{\omega}v_0$, if we simultaneously parametrize by unit length w.r.t. to the Angenent half-plane metric, to obtain:
\beq\label{Jacobi}
\tilde{v}_0''+K_\textrm{Ang} \tilde{v}_0=\frac{1}{\sqrt\omega}\mathcal{L}_0v_0=0.
\eeq
That is, of course, the well-known Jacobi equation for the geodesic $\gamma$. Since $\sqrt{\omega}>0$, Dirichlet conditions for $v_0$ correspond exactly to Dirichlet conditions for $\tilde{v}_0$. Note that since, by symmetry,
\beq
\frac{\partial\sqrt{\omega}}{\partial t}(t_0)=0,\quad\textrm{when}\quad \gamma(t_0)\in\{x_1=0\},
\eeq
imposing Neumann conditions on $\{x_1=0\}$ is also equivalent for $v_0$ and $\tilde{v}_0$.

\begin{figure}
\centering
\includegraphics[width=0.9\textwidth]{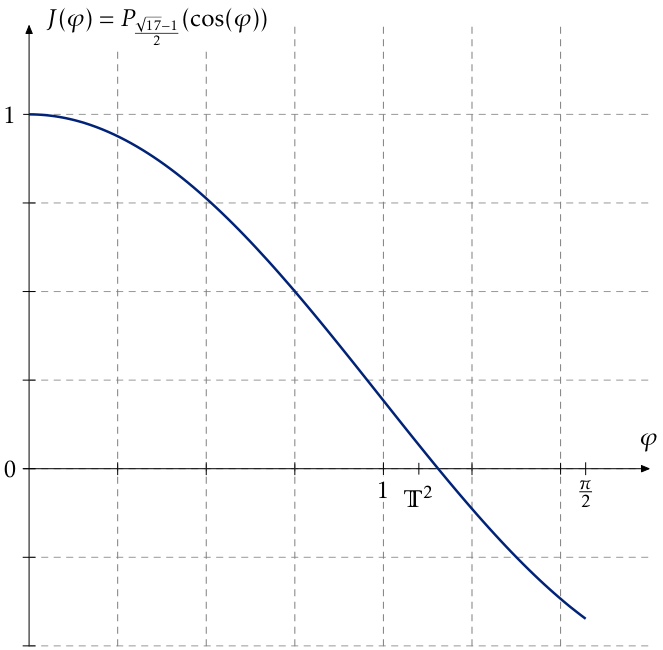}
\caption{Jacobi field $J$ on $S_3\subseteq\mathbb{S}^2$, with Neumann condition on $\{x=0\}$. Drawn w.r.t. polar coordinates around $(0,1)$, where $\varphi$ is the angle out from the $x_2$-axis in Figure \ref{TorusAndSphere}. Point of intersection with $\mathbb{T}^2$ indicated.}
\end{figure}

\begin{figure}
\centering
\includegraphics[width=0.9\textwidth]{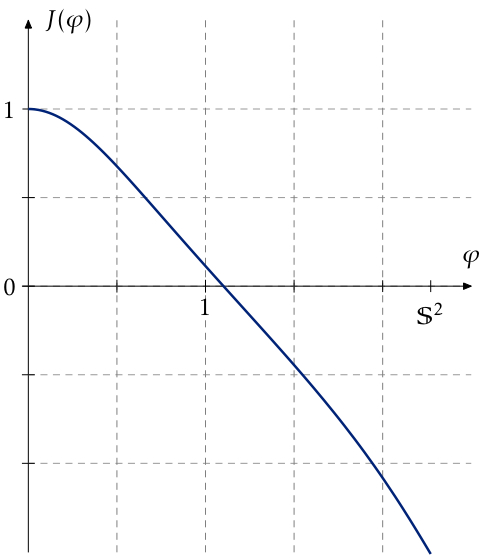}
\caption{Jacobi field $J$ on $S_5\subseteq\mathbb{T}^2$, with Neumann condition on $\{x=0\}$. Drawn w.r.t. polar coordinates around $(0,1)$, where $\varphi$ is the angle out from the $x_2$-axis in Figure \ref{TorusAndSphere}. Point of intersection with $\mathbb{S}^2$ indicated.}
\end{figure}

\begin{figure}
\centering
\includegraphics[width=0.9\textwidth]{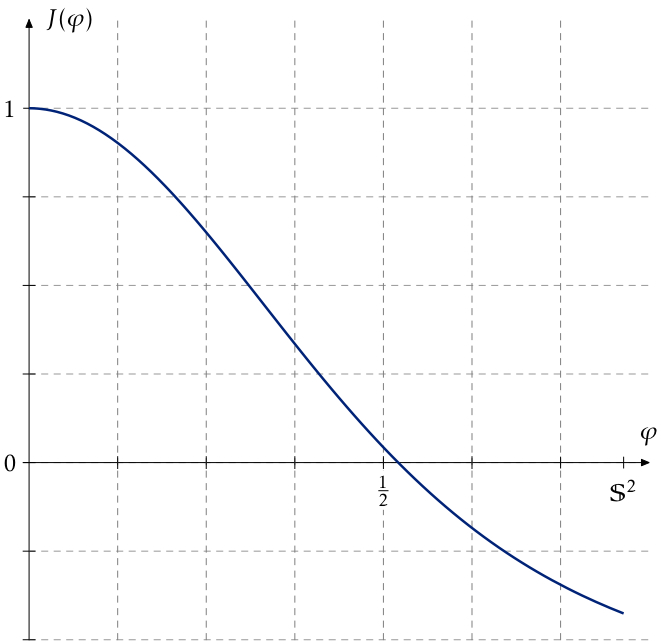}
\caption{Jacobi field $J$ on $S_6\subseteq\mathbb{T}^2$, with Neumann condition on $\{x=0\}$. Drawn w.r.t. polar coordinates around $(0,1)$, where $\varphi$ is the angle out from the $x_2$-axis in Figure \ref{TorusAndSphere}. Point of intersection with $\mathbb{S}^2$ indicated.}
\end{figure}

\begin{lemma}\label{GraphJacobi}
For graphs of the form $(x_1,u(x_1))$ the operator $\mathcal{L}_0$ specializes to $\omega=1+(u')^2$, and
\begin{align*}
\omega\mathcal{L}_0v
&=v''-\frac{x_1}{2}\left(1+(u')^2\right)v'+\left[\left(\frac{u-x_1u'}{2}-\frac{1}{u}\right)^2+\frac{1}{u^2}+\frac{1+(u')^2}{2}\right]v,\\
&=:v''+P(x_1,u,u')v'+Q(x_1,u,u')v,
\end{align*}
when $u$ is a solution to the shrinker equation.

For solution graphs of the form $(f(x_2),x_2)$ we have $\omega=1+(f')^2$ and the formula is:
\begin{align*}
\omega\mathcal{L}_0g&=g''+\left[\frac{1}{x_2}-\frac{x_2}{2}\right]\left(1+(f')^2\right)g'+\left[\left(\frac{x_2f'-f}{2}-\frac{f'}{x_2}\right)^2+\frac{(f')^2}{x_2^2}+\frac{1+(f')^2}{2}\right]g\\
&=:g''+R(x_2,f,f')g'+S(x_2,f,f')g.
\end{align*}
\end{lemma}

Let us as a preliminary consideration note that the mean curvature $H$ has rotational symmetry, and is an eigenfunction with eigenvalue $1$ (see \cite{CM2}):
\beq\label{SturmLi}
\mathcal{L}H=H,
\eeq
and Neumann conditions on $\{x_1=0\}$. The profile of $\mathbb{T}^2$ is convex as shown in \cite{KM}, and thus since the sign of the mean curvature changes only at points of tangential contact with a straight line from the origin, we see that on $\mathbb{T}^2$ the function $H\circ \gamma$ has exactly one zero.

By Sturm-Liouville theory, we now conclude from (\ref{SturmLi}) that a solution to the Neumann problem for $\mathcal{L}_0$, if it exists, needs to have at least two zeros in the interval. However, it of course turns out there are no such non-trivial fields with Neumann conditions, which is what we now will apply more detailed analysis to show.

\subsubsection{Surfaces contained in $\mathbb{S}^2$}
For surfaces contained in $\mathbb{S}^2$ it is of natural convenience to use polar coordinates. Recall that a curve $(\rho,\varphi), \varphi = \arctan (x/ z)$ in the $x z$-plane generated by a function $\rho(\varphi)$  that generates a smooth solution to the self-shrinker equation (\ref{SSEqToru}) satisfies:
\beq
\rho''(\varphi)=\frac{1}{\rho}\left\{\rho^2+2(\rho')^2+\Big[1-\frac{\rho^2}{2}-\frac{\rho'}{\rho\tan\varphi}\Big]\big(\rho^2+(\rho')^2\big)\right\}.
\eeq

A function $w$ giving a (via a unit normal w.r.t. Euclidean length) variation field, must thus on $\mathbb{S}^2$ satisfy the equation (see Appendix A in \cite{KKM})
\beq\label{PolLin}
w''+\frac{1}{\tan\varphi}w'+4w=0,
\eeq
with appropriate boundary conditions. The substitution $x=\cos(\varphi)$ in (\ref{PolLin}) gives Legendre's differential equation, and the solution is:
\[
w(\varphi)=C_1 P_{l_0}(\cos \varphi)+C_2Q_{l_0}(\cos\varphi),
\]
where $P_l$ and $Q_l$ are respectively the Legendre functions of the first and second kind, and $l_0=(\sqrt{17}-1)/2$ is the positive root of $l(l+1)=4$.

For the surface $S_1\subseteq \mathbb{S}^2$, which is generated by rotation of the radius 2 quarter-circle, the boundary conditions in the theorem are $w'(0)=0$ and $w'(\fracsm{\pi}{2})=0$. But since $Q_{l_0}(\cos\varphi)$ has a pole at $\varphi=0$, we see that $C_2=0$. Thus, if we normalize $w$ so that $C_1=1$, we have by the first condition that $w=P_{l_0}(\cos \varphi)$. However,
\beq
\frac{dP_{l_0}(\cos\varphi)}{d\varphi}_{\mid \varphi=\fracsm{\pi}{2}}=\frac{\sqrt{\pi}}{2}\frac{\sqrt{17}+1}{\Gamma\left(\frac{1-\sqrt{17}}{4}\right)\Gamma\left(\frac{5+\sqrt{17}}{4}\right)}\neq 0,
\eeq
and hence there is no such Neumann mode. By the preceding, we therefore conclude that under $N$-fold symmetry, for a large enough $N>0$,
\beq
\ker\mathcal{L}_{S_1}=\{0\},\quad\textrm{on}\quad S_1=\mathbb{S}^2\cap\{x_1\geq 0\}\quad \textrm{[Neumann conditions]}.
\eeq

As for the surface $S_3\subseteq\mathbb{S}^2$, which we let for definiteness be the component such that $(2,0)\in S_3$, again $C_1=1$ and $C_2=0$. Now, since by (\ref{SphereCrossAngle}),
\beq
P_{l_0}(\cos(\angle(\vec{e_1},p^+)))\geq P_{l_0}(\cos(\fracsm{11}{10}+\fracsm{1}{20}))\geq \fracsm{1}{200}>0,
\eeq
we conclude again
\beq
\ker\mathcal{L}_{S_3}=\{0\}\quad\textrm{[Dirichlet conditions]}.
\eeq

For $S_4\subseteq\mathbb{S}^2$, the component with $(0,2)\in S_4$, we see that with $w'(\fracsm{\pi}{2})=0$ and fixing $w(\fracsm{\pi}{2})=1$,
\[
\begin{pmatrix}
C_1\\
C_2
\end{pmatrix}
=
\begin{pmatrix}
P_{l_0}(\cos(\fracsm{\pi}{2})) & Q_{l_0}(\cos(\fracsm{\pi}{2}))\\
\fracsm{d}{d\varphi}P_{l_0}(\cos(\fracsm{\pi}{2})) & \fracsm{d}{d\varphi} P_{l_0}(\cos(\fracsm{\pi}{2}))
\end{pmatrix}^{-1}
\begin{pmatrix}
1\\0
\end{pmatrix}
\]
such that with these constants
\[
w(\angle(\vec{e_1},p^+))\geq w(\fracsm{11}{10})-100\varepsilon_{\textrm{gap}}>\fracsm{1}{2}>0.
\]
Thus on the surface $S_4$ with Neumann and Dirichlet conditions as in the statement of the theorem, we also conclude
\beq
\ker\mathcal{L}_{S_4}=\{0\}\quad\textrm{[Neumann on $\{x_1=0\}$, Dirichlet on $S_4\cap\mathbb{T}^2$]}.
\eeq

\subsubsection{Surfaces Contained in $\mathbb{T}^2$}
In this section we will show that the Jacobi fields with Neumann conditions $u_0'(0)=0$ (and normalized to $u_0(0)=1$), propagated from respectively the top and the bottom of the torus $\mathbb{T}^2$ from Section \ref{sec:torus} have the following end-point values at the point where $\mathbb{T}^2$ intersects the round 2-sphere of radius $2$:
\beq\label{topmatrix}
\begin{pmatrix}
u_0(t_\textrm{top})\\
u_0'(t_\textrm{top})
\end{pmatrix}
=\begin{pmatrix}
-\fracsm{22}{50}\pm 10\varepsilon_{\mathrm{gap}}\\
\null \\
-\fracsm{37}{50}\pm 10\varepsilon_{\mathrm{gap}}
\end{pmatrix}
\eeq
and
\beq\label{botmatrix}
\begin{pmatrix}
u_0(t_\textrm{bot})\\
u_0'(t_\textrm{bot})
\end{pmatrix}
=
\begin{pmatrix}
-\fracsm{77}{50}\pm 10\varepsilon_{\mathrm{gap}}\\
\null \\
-\fracsm{84}{50}\pm 10\varepsilon_{\mathrm{gap}},
\end{pmatrix}
\eeq
where the notation is meant to imply that each component is contained in the corresponding intervals arising from both choices of sign.

This means firstly that the Dirichlet problems on each part (w/ Neumann conditions on $\mathcal{P}$) have trivial kernel, that is:
\begin{align}
\ker\mathcal{L}_{S_5}=\{0\}\quad\textrm{[Neumann on $\{x_1=0\}$, Dirichlet on $S_5\cap\mathbb{S}^2$]},\\
\ker\mathcal{L}_{S_6}=\{0\}\quad\textrm{[Neumann on $\{x_1=0\}$, Dirichlet on $S_6\cap\mathbb{S}^2$]}.
\end{align}

Secondly, note that non-triviality of the kernel of $\mathcal{L}$ on $\mathbb{T}^2\cap\{x_1\geq0\}$ with Neumann conditions has now been reduced to the conditions
\begin{align*}
&\alpha u_0(t_\textrm{top}) =\beta u_0(t_\textrm{top}),\\
&\alpha u_0'(t_\textrm{top}) =-\beta u_0'(t_\textrm{bot}),
\end{align*}
for a non-zero pair $(\alpha,\beta)$ or in other words, singularity of the matrix $\mathcal{N}$:
\beq
\mathcal{N}:=
\begin{pmatrix}
u_0(t_\textrm{top}) & -u_0(t_\textrm{bot})\\
u_0'(t_\textrm{top}) & u_0'(t_\textrm{bot})
\end{pmatrix}.
\eeq
But in fact from (\ref{topmatrix})--(\ref{botmatrix}) we see
\beq
\det\mathcal{N}\geq \fracsm{9}{5} >0,
\eeq
so we finally conclude that also
\[
\ker\mathcal{L}_{S_2}=\{0\}\quad \textrm{[Neumann on $\{x_1=0\}$}.
\]

To show the required estimates of the Jacobi fields, we consider first an approximate solution $V$ to the linearized equation on the approximate curve $\Gamma$ from the previous section.

\underline{Top: $x_1$-graph (Jacobi Equation)}

Let us consider the first part of $\Gamma$, graphical over the $x_1$-axis, on $x_1\in [0,\fracsm{3}{5}]$. Here, we have the expressions (recall Lemma \ref{GraphJacobi} for the definition of $P$ an $Q$).
\begin{align*}
&\left|\partial_2P(x_1,\xi,U')\right|=0,\\
&\left|\partial_3P(x_1,u,\xi')\right|=|x_1 \xi'|,\\
&\left|\partial_2Q(x_1,\xi,U')\right|=2\left|\left(\frac{\xi-x_1U'}{2}-\frac{1}{\xi}\right)\left(\frac{1}{2}+\frac{1}{\xi^2}\right)-\frac{1}{\xi^3}\right|,\\
&\left|\partial_3Q(x_1,u,\xi')\right|=\left|x_1\left(\frac{u-x_1\xi'}{2}-\frac{1}{u}\right)-\xi'\right|.
\end{align*}

Assume for a small $\delta_1^T>0$ the uniform bounds:
\begin{align}
&\left|V''+P(x_1,U,U')V'+Q(x_1,U,U')V\right|\leq \delta_1^T,\\
&\left|\partial_3P(x_1,u,\xi')\right||v'|+\left|\partial_3Q(x_1,u,\xi')\right||v|\leq x_1(2x_1^2+3),\\
&\left|\partial_2 Q(x_1,\xi,U')\right||v|\leq \fracsm{8}{5}-\fracsm{3}{2}x_1^2.
\end{align}

We let $\Phi(x_1):=|V'(x_1)-v'(x_1)|$ and estimate: 
\begin{align*}
\Phi'(x_1)&\leq \left||V'-v'|'\right|\overset{\textrm{a.e}}{=}|V''-v''|\\
&\leq \left|P(x_1,U,U')V' - P(x_1,u,u')v'\right|+\left|Q(x_1,U,U')V- Q(x_1,u,u')v\right|+\delta_1^T\\
&\leq \left|P(x_1,U,U')\right||V' -v'|+\left|P(x_1,U,U') - P(x_1,u,u')\right||v'|\\
&\qquad+\left|Q(x_1,U,U')\right||V -v|+\left|Q(x_1,U,U') - Q(x_1,u,u')\right||v|+\delta_1^T\\
&\leq \left|P(x_1,U,U')\right||V' -v'|+\left|Q(x_1,U,U')\right||V -v|\\
&\qquad +\left(\left|\partial_3P(x_1,u,\xi')\right||v'|+\left|\partial_3Q(x_1,u,\xi')\right||v|\right)|U'-u'|\\
&\qquad +\left(\cancel{\left|\partial_2 P(x_1,\xi,U')\right|}|v'|+\left|\partial_2 Q(x_1,\xi,U')\right||v|\right)|U-u|+\delta_1^T\\
&\leq \left|P(x_1,U,U')\right|\Phi(x_1) + \left|Q(x_1,U,U')\right|\int_{0}^{x_1}\Phi(s)ds+\left|Q(x_1,U,U')\right||V(0)-v(0)|\\
&\qquad+x_1(2x_1^2+3)|U'(x_1)-u'(x_1)|+(\fracsm{8}{5}-\fracsm{3}{2}x_1^2)|U(x_1)-u(x_1)|+\delta_1^T.
\end{align*}
We integrate on $[0,x_1]$ to get:
\begin{align*}
\Phi(x_1)\leq &\int_0^{x_1}\left[\left|P(s,U,U')\right|+\int_0^{x_1}\left|Q(t,U,U')\right|dt\right]\Phi(s)ds+\left(\int_0^{x_1}\left|Q(s,U,U')\right|ds\right)|V(0)-v(0)|\\
&\quad +\int_0^{x_1}s(2s^2+3)|U'(s)-u'(s)|ds+\int_0^{x_1}(\fracsm{8}{5}-\fracsm{3}{2}s^2)|U(s)-u(s)|ds+\delta^T_1x_1.
\end{align*}

Recall the estimates for $|U'(s)-u'(s)|$, which lead to:
\begin{align*}
\int_0^{x_1}s(2s^2+3)|U'(s)-u'(s)|ds&\leq\left[\fracsm{4}{5}|U(0)-u(0)|+\varepsilon_1^T\right]\int_0^{x_1}s^2(2s^2+3)\exp\Big\{\fracsm{52}{15}s^2\Big\}ds\\
&\leq\fracsm{9}{20}|U(0)-u(0)|+\fracsm{14}{25}\varepsilon^T_1.
\end{align*}
Note also that from the estimates for $|U-u|$ from $|U'-u'|$, we have already once estimated the integral of the latter. We now need the sizes of these elementary Gaussian double integrals:
\begin{align*}
&\int_0^{\fracsm{3}{5}}(\fracsm{8}{5}-\fracsm{3}{2}x_1^2)\left(1+\int_0^{x_1}\fracsm{4s}{5}\exp\Big\{\fracsm{52}{15}s^2\Big\}ds\right)dx_1\leq\frac{23}{25},\\
&\int_0^{\fracsm{3}{5}}(\fracsm{8}{5}-\fracsm{3}{2}x_1^2)\int_0^{x_1}s\exp\Big\{\fracsm{52}{15}s^2\Big\}dsdx_1\leq\frac{7}{100}.
\end{align*}
Thus
\[
\int_0^{\fracsm{3}{5}}(\fracsm{8}{5}-\fracsm{3}{2}s^2)|U(s)-u(s)|ds\leq \fracsm{23}{25}|U(0)-u(0)|+\fracsm{7}{100}\varepsilon^T_1.
\]

Now, we will furthermore assume the following bounds on the test functions, pertaining to the approximation by $\varepsilon$-geodesics:
\begin{align}
&\int_0^{x_1}|P(s,U,U')|ds\leq \int_0^{x_1} s(\fracsm{7}{5}s^2+\fracsm{7}{10})ds\leq \fracsm{7}{20}x_1^2(x_1^2+1),\\
&\int_0^{x_1}\left|Q(s,U,U')\right|ds\leq \int_0^{x_1}(\fracsm{16}{5}s^2 + \fracsm{5}{2})ds\leq \fracsm{16}{15}x_1^3 + \fracsm{5}{2}x_1.
\end{align}

Applying Gr\"{o}nwall-Bellman again, to these new estimates, we see, using also that $V(0)=v(0)=1$ by assumption:
\begin{align*}
\Phi(x_1)\leq &\left[\int_0^{x_1}s(2s^2+3)|U'(s)-u'(s)|ds+\int_0^{x_1}(\fracsm{8}{5}-\fracsm{3}{2}s^2)|U(s)-u(s)|ds+\delta^T_1x_1\right]\times\\
&\qquad\exp\left\{\int_0^{x_1}\left|P(s,U,U')\right|ds+x_1\int_0^{x_1}\left|Q(s,U,U')\right|ds\right\}\\
\leq & \left[\fracsm{137}{100}|U(0)-u(0)|+\fracsm{63}{100}\varepsilon^T_1+\fracsm{3}{5}\delta^T_1\right]\times\\
&\qquad\exp\left\{\fracsm{7}{20}x_1^2(x_1^2+1)+\fracsm{16}{15}x_1^4+\fracsm{5}{2}x_1^2\right\}.
\end{align*}
Thus
\begin{align*}
|V'(\fracsm{3}{5})-v'(\fracsm{3}{5})|\leq \fracsm{23}{5}|U(0)-u(0)|+\fracsm{11}{5}\varepsilon_1^T+\fracsm{51}{25}\delta_1^T,\\
|V(\fracsm{3}{5})-v(\fracsm{3}{5})|\leq \fracsm{13}{10}|U(0)-u(0)|+\fracsm{3}{5}\varepsilon_1^T+\fracsm{11}{50}\delta_1^T.
\end{align*}
Here, the last estimate followed by integrating the estimates above, and using again that $V(0)=v(0)$.

\underline{Top: $x_2$-graph to the sphere (Jacobi Equation)}

We consider the next part of $\Gamma$, graphical over the $x_2$-axis in the region $[y_{\mathbb{S}^2},u_T(3/5)]$ (in the backwards direction). Here, we have the expressions:
\begin{align*}
&\left|\partial_2R(x_2,\xi,F')\right|=0,\\
&\left|\partial_3R(x_2,f,\xi')\right|=2\left|\frac{1}{x_2}-\frac{x_2}{2}\right||\xi'|,\\
&\left|\partial_2S(x_2,\xi,F')\right|=\left|\left(\frac{1}{x_2}-\frac{x_2}{2}\right)F'+\frac{\xi}{2}\right|,\\
&\left|\partial_3S(x_2,f,\xi')\right|=\left|\left(\frac{1}{x_2}-\frac{x_2}{2}\right)(x_2\xi'-f)+\frac{4\xi'}{x_2^2}\right|.
\end{align*}
Assume for a small $\delta_2^T>0$ the uniform bounds:
\begin{align}
&\left|G''+R(x_2,F,F')G'+S(x_2,F,F')G\right|\leq \delta_2^T,\\
&\left|\partial_3R(x_2,u,\xi')\right||g'|+\left|\partial_3S(x_2,f,\xi')\right||g|\leq \fracsm{3}{10}+\fracsm{33}{20}\left(x_2-y_{\mathbb{S}^2}\right)^5 \\
&\left|\partial_2 S(x_2,\xi,U')\right||g|\leq \eta(x_2),
\end{align}
where
\beq
\eta(x_2)=
\begin{cases}
\fracsm{41}{200}-\frac{3}{10}\left(x_2-y_{\mathbb{S}^2}\right)^2,\quad &x_2\in [y_{\mathbb{S}^2},\fracsm{5}{2}],\\
\fracsm{3}{4}-\fracsm{5}{2}\left(x_2-y_{\mathbb{S}^2}\right)+\fracsm{21}{10}\left(x_2-y_{\mathbb{S}^2}\right)^2,\quad &x_2\in [\fracsm{5}{2},u_T(3/5)].
\end{cases}
\eeq

We let $\Psi(x_2):=|G'(x_2)-g'(x_2)|$ and estimate:
\begin{align*}
\Psi'(x_2)&\leq \left||G'-g'|'\right|\overset{\textrm{a.e}}{=}|G''-g''|\\
&\leq \left|R(x_2,F,F')G' - R(x_2,f,f')g'\right|+\left|S(x_2,F,F')G- S(x_2,f,f')g\right|+\delta_2^T\\
&\leq \left|R(x_2,F,F')\right||G' -g'|+\left|R(x_2,F,F') - R(x_2,f,f')\right||g'|\\
&\qquad+\left|S(x_2,F,F')\right||G -g|+\left|S(x_2,F,F') - S(x_2,f,f')\right||g|+\delta_2^T\\
&\leq \left|R(x_2,F,F')\right||G' - g'|+\left|S(x_2,F,F')\right||G - g|\\
&\qquad +\left(\left|\partial_3R(x_2,f,\xi')\right||g'|+\left|\partial_3S(x_2,f,\xi')\right||g|\right)|F'-f'|\\
&\qquad +\left|\partial_2 S(x_2,\xi,F')\right||g||F-f|+\delta_2^T\\
&\leq \left|R(x_2,F,F')\right|\Psi(x_2) + \left|S(x_2,F,F')\right|\int_{u_T(3/5)}^{x_2}\Psi(s)ds+\left|S(x_2,F,F')\right||G(u_T(3/5))-g(u_T(3/5))|\\
&\qquad+\left[\fracsm{3}{10}+\fracsm{33}{20}\left(x_2-y_{\mathbb{S}^2}\right)^5\right]|F'(x_2)-f'(x_2)|+\eta(x_2)|F(x_2)-f(x_2)|+\delta_2^T.
\end{align*}

We integrate on $[x_2,u_T(3/5)]$ to get:
\begin{align*}
\Psi(x_2)\leq &\int_{x_2}^{u_T(3/5)}\left[\left|R(s,F,F')\right|+\int_{x_2}^{u_T(3/5)}\left|S(t,F,F')\right|dt\right]\Psi(s)ds\\
&\quad+\left(\int_{x_2}^{u_T(3/5)}\left|S(s,F,F')\right|ds\right)|G(u_T(3/5))-g(u_T(3/5))|\\
&\quad +\int_{x_2}^{u_T(3/5)}\left[\fracsm{3}{10}+\fracsm{33}{20}\left(s-y_{\mathbb{S}^2}\right)^5\right]|F'(s)-f'(s)|ds+\int_{x_2}^{u_T(3/5)}\eta(s)|F(s)-f(s)|ds\\
&\quad+|G'(u_T(3/5))-g'(u_T(3/5))|+\delta_2^T(u_T(3/5)-x_2).
\end{align*}

Recall, we have above shown the estimates:
\begin{align*}
&|F'(x_2)-f'(x_2)|\leq \Big[(\fracsm{13}{8}-\fracsm{x_2}{2})|F(u_T(3/5))-f(u_T(3/5))|+|F'(u_T(3/5))-f'(u_T(3/5))|\\
&\qquad\qquad\qquad\qquad\qquad+\varepsilon_2^T(u_T(3/5)-x_2)\Big]\exp\left\{\fracsm{7}{4}-\fracsm{9}{10}\left(x_2-y_{\mathbb{S}^2}\right)^2+\left(\fracsm{13}{8}-\fracsm{x_2}{2}\right)\left(u_T(3/5)-x_2\right)\right\},\\
&|F(u_T(3/5))-f(u_T(3/5))|\leq\fracsm{10}{12}\left(\fracsm{9}{25}\varepsilon_1^T+\fracsm{23}{80}|U(0)-u(0)|\right),\\
&|F'(u_T(3/5))-f'(u_T(3/5))|\leq\left(\fracsm{10}{12}\right)^2(\varepsilon_1^T+\fracsm{4}{5}|U(0)-u(0)|).
\end{align*}

We therefore get the bound:
\begin{align*}
\int_{y_{\mathbb{S}^2}}^{u_T(3/5)}\left[\fracsm{3}{10}+\fracsm{33}{20}\left(s-y_{\mathbb{S}^2}\right)^5\right]|F'(s)-f'(s)|ds\leq \fracsm{80}{25}|U(0)-u(0)| + 6\varepsilon_1^T +\fracsm{27}{10}\varepsilon_2^T.
\end{align*}

Again, we will need the sizes of some elementary Gaussian double integrals:
\begin{align*}
&\int_{y_{\mathbb{S}^2}}^{u_T(3/5)}\eta(x_2)\left(1+\int_{x_2}^{u_T(3/5)}(\fracsm{13}{8}-\fracsm{s}{2})\exp\left\{\fracsm{7}{4}-\fracsm{9}{10}\left(s-y_{\mathbb{S}^2}\right)^2+\left(\fracsm{13}{8}-\fracsm{s}{2}\right)\left(u_T(3/5)-s\right)\right\}ds\right)dx_2\leq\fracsm{29}{50}\\
&\int_{y_{\mathbb{S}^2}}^{u_T(3/5)}\int_{x_2}^{u_T(3/5)}\eta(x_2)\exp\left\{\fracsm{7}{4}-\fracsm{9}{10}\left(s-y_{\mathbb{S}^2}\right)^2+\left(\fracsm{13}{8}-\fracsm{s}{2}\right)\left(u_T(3/5)-s\right)\right\}dsdx_2\leq \fracsm{29}{50},\\
&\int_{y_{\mathbb{S}^2}}^{u_T(3/5)}\int_{x_2}^{u_T(3/5)}\eta(x_2)(u_T(3/5)-s)\exp\left\{\fracsm{7}{4}-\fracsm{9}{10}\left(s-y_{\mathbb{S}^2}\right)^2+\left(\fracsm{13}{8}-\fracsm{s}{2}\right)\left(u_T(3/5)-s\right)\right\}dsdx_2\leq\fracsm{19}{50}.
\end{align*}
Thus
\[
\int_{y_{\mathbb{S}^2}}^{u_T(3/5)}\eta(x_2)|F(x_2)-f(x_2)|dx_2\leq \fracsm{1}{2}|U(0)-u(0)|+\fracsm{29}{50}\varepsilon_1^T+\fracsm{19}{50}\varepsilon_2^T.
\]

Assume once again bounds for the $\varepsilon$-geodesics:
\begin{align}
&\int_{x_2}^{u_T(3/5)}|R(s,F,F')|ds\leq -\fracsm{16}{25}x_2^2 + \fracsm{22}{10}x_2 - \fracsm{18}{25},\\
&\int_{x_2}^{u_T(3/5)}\left|S(s,F,F')\right|ds\leq -\fracsm{27}{100}x_2^2+\fracsm{7}{25}x_2 + \fracsm{19}{10}.
\end{align}

By Gr\"{o}nwall-Bellman with the new estimates, we see:
\begin{align*}
\Psi(x_2)\leq &\Bigg[\left(\int_{x_2}^{u_T(3/5)}\left|S(s,F,F')\right|ds\right)|G(u_T(3/5))-g(u_T(3/5))|\\
&\qquad+\int_{x_2}^{u_T(3/5)}\left[\fracsm{3}{10}+\fracsm{33}{20}\left(x_2-y_{\mathbb{S}^2}\right)^5\right]|F'(s)-f'(s)|ds\\
&\qquad+\int_{x_2}^{u_T(3/5)}\eta(x_2)|F(s)-f(s)|ds+\delta^T_2(u_T(3/5)-x_2)+|G'(u_T(3/5))-g'(u_T(3/5))|\Bigg]\times\\
&\quad\exp\left\{\int_{x_2}^{u_T(3/5)}\left|R(s,F,F')\right|ds+(u_T(3/5)-x_2)\int_{x_2}^{u_T(3/5)}\left|S(s,F,F')\right|ds\right\}\\
\leq &24\Big[\fracsm{31}{20}\left(\fracsm{13}{10}|U(0)-u(0)|+\fracsm{3}{5}\varepsilon_1^T+\fracsm{11}{50}\delta_1^T\right) + \fracsm{80}{25}|U(0)-u(0)| + 6\varepsilon_1^T +\fracsm{27}{10}\varepsilon_2^T\\
&\quad +\fracsm{1}{2}|U(0)-u(0)|+\fracsm{29}{50}\varepsilon_1^T+\fracsm{19}{50}\varepsilon_2^T+ (\fracsm{49}{16} - \fracsm{41}{23})\delta^T_2+\frac{\fracsm{23}{5}|U(0)-u(0)|+\fracsm{11}{5}\varepsilon_1^T+\fracsm{51}{25}\delta_1^T}{|u'(\fracsm{3}{5})|}\Big].
\end{align*}

As before, we thus have the estimate:

\begin{align*}
|\Psi(y_{\mathbb{S}^2})|=|G'(y_{\mathbb{S}^2})-g'(y_{\mathbb{S}^2})|\leq 237|U(0)-u(0)|+228\varepsilon_1^T+74\varepsilon_2^T+53\delta_1^T+ 31\delta^T_2,
\end{align*}
where we used $1/|u'(\fracsm{3}{5})|\leq \fracsm{9}{10}$. By integration, we can again accurately estimate the $\|\cdot\|_{C^0}$-norm, although here it suffices to use a simple supremum bound on the integrand:
\begin{align*}
|G(y_{\mathbb{S}^2})-g(y_{\mathbb{S}^2})|&\leq |G(u_T(3/5))-g(u_T(3/5))|+|y_{\mathbb{S}^2}-u_T(3/5)||\Psi(y_{\mathbb{S}^2})|\\
&=\fracsm{13}{10}|U(0)-u(0)|+\fracsm{3}{5}\varepsilon_1^T+\fracsm{11}{50}\delta_1^T+\fracsm{13}{10}\left(237|U(0)-u(0)|+228\varepsilon_1^T+74\varepsilon_2^T+53\delta_1^T+ 31\delta^T_2\right)\\
&\leq310|U(0)-u(0)|+297\varepsilon_1^T+97\varepsilon_2^T+70\delta_1^T+ 41\delta^T_2.
\end{align*}

\underline{Bottom: $x_1$-graph (Jacobi Equation)}

Let us consider the first part of $\Gamma$, graphical over the $x_1$-axis. Here, we have the expressions:
\begin{align*}
&\left|\partial_2P(x_1,\xi,U')\right|=0,\\
&\left|\partial_3P(x_1,u,\xi')\right|=|x_1 \xi'|,\\
&\left|\partial_2Q(x_1,\xi,U')\right|=2\left|\left(\frac{\xi-x_1U'}{2}-\frac{1}{\xi}\right)\left(\frac{1}{2}+\frac{1}{\xi^2}\right)-\frac{1}{\xi^3}\right|,\\
&\left|\partial_3Q(x_1,u,\xi')\right|=\left|x_1\left(\frac{u-x_1\xi'}{2}-\frac{1}{u}\right)-\xi'\right|.
\end{align*}

Assume for a small $\delta_1^B>0$ the uniform bounds:
\begin{align}
&\left|V''+P(x_1,U,U')V'+Q(x_1,U,U')V\right|\leq \delta_1^B,\\
&\left|\partial_3P(x_1,u,\xi')\right||v'|+\left|\partial_3Q(x_1,u,\xi')\right||v|\leq 2,\\
&\left|\partial_2 Q(x_1,\xi,U')\right||v|\leq 41-80x_1.
\end{align}

We let $\Phi(x_1):=|V'(x_1)-v'(x_1)|$ and estimate:
\begin{align*}
\Phi'(x_1)&\leq \left||V'-v'|'\right|\overset{\textrm{a.e}}{=}|V''-v''|\\
&\leq \left|P(x_1,U,U')V' - P(x_1,u,u')v'\right|+\left|Q(x_1,U,U')V- Q(x_1,u,u')v\right|+\delta_1^B\\
&\leq \left|P(x_1,U,U')\right||V' -v'|+\left|P(x_1,U,U') - P(x_1,u,u')\right||v'|\\
&\qquad+\left|Q(x_1,U,U')\right||V -v|+\left|Q(x_1,U,U') - Q(x_1,u,u')\right||v|+\delta_1^B\\
&\leq \left|P(x_1,U,U')\right||V' -v'|+\left|Q(x_1,U,U')\right||V -v|\\
&\qquad +\left(\left|\partial_3P(x_1,u,\xi')\right||v'|+\left|\partial_3Q(x_1,u,\xi')\right||v|\right)|U'-u'|\\
&\qquad +\left|\partial_2 Q(x_1,\xi,U')\right||v||U-u|+\delta_1^B\\
\end{align*}
Hence:
\begin{align*}
&\leq \left|P(x_1,U,U')\right|\Phi(x_1) + \left|Q(x_1,U,U')\right|\int_{0}^{x_1}\Phi(s)ds+\left|Q(x_1,U,U')\right||V(0)-v(0)|\\
&\qquad+2|U'(x_1)-u'(x_1)|+(41-80x_1)|U(x_1)-u(x_1)|+\delta_1^B.
\end{align*}
We integrate on $[0,x_1]$ to get:
\begin{align*}
\Phi(x_1)\leq &\int_0^{x_1}\left[\left|P(s,U,U')\right|+\int_0^{x_1}\left|Q(t,U,U')\right|dt\right]\Phi(s)ds+\left(\int_0^{x_1}\left|Q(s,U,U')\right|ds\right)|V(0)-v(0)|\\
&\quad +2\int_0^{x_1}|U'(s)-u'(s)|ds+\int_0^{x_1}(41-80s)|U(s)-u(s)|ds+\delta^B_1x_1.
\end{align*}

Recall the estimates:
\begin{align*}
2\int_0^{x_1}|U'(s)-u'(s)|ds&\leq2|U(0)-u(0)|\int_0^{x_1}(\fracsm{29}{5}s+\fracsm{1}{20})\exp\Big\{\fracsm{49}{5}s^2+\fracsm{s}{4}\Big\}ds\\
&\qquad+2\varepsilon^B_1\int_0^{x_1}s\exp\Big\{\fracsm{49}{5}s^2+\fracsm{s}{4}\Big\}ds\\
&\leq\fracsm{36}{5}|U(0)-u(0)|+\fracsm{6}{5}\varepsilon^B_1.
\end{align*}
Note also that from the estimates for $|U-u|$ from $|U'-u'|$, we have already once estimated the integral of the latter. We now need the sizes of these elementary Gaussian double integrals:
\begin{align*}
&\int_0^{\fracsm{1}{2}}(41-80x_1)\left(1+\int_0^{x_1}(\fracsm{29}{5}s+\fracsm{1}{20})\exp\Big\{\fracsm{49}{5}s^2+\fracsm{s}{4}\Big\}ds\right)dx_1\leq\frac{67}{5},\\
&\int_0^{\fracsm{1}{2}}\int_0^{x_1}(41-80x_1)s\exp\Big\{\fracsm{49}{5}s^2+\fracsm{s}{4}\Big\}dsdx_1\leq\frac{12}{25}.
\end{align*}
Thus
\[
\int_0^{\fracsm{1}{2}}(41-80x_1)|U(x_1)-u(x_1)|dx_1\leq \fracsm{67}{5}|U(0)-u(0)|+\frac{12}{25}\varepsilon^B_1.
\]

Now, we will furthermore assume the bounds pertaining to the $\varepsilon$-geodesics:
\begin{align}
&\int_0^{x_1}|P(s,U,U')|ds\leq \frac{4}{9}x_1^2,\\
&\int_0^{x_1}\left|Q(s,U,U')\right|ds\leq 4-16(x_1-\fracsm{1}{2})^2.
\end{align}

Again, by Gr\"{o}nwall-Bellman we see (with $V(0)=v(0)$):
\begin{align*}
\Phi(x_1)\leq &\left[2\int_0^{x_1}|U'(s)-u'(s)|ds+\int_0^{x_1}(41-80s)|U(s)-u(s)|ds+\delta^B_1x_1\right]\times\\
&\qquad\exp\left\{\int_0^{x_1}\left|P(s,U,U')\right|ds+x_1\int_0^{x_1}\left|Q(s,U,U')\right|ds\right\}\\
\leq & \left[\fracsm{103}{5}|U(0)-u(0)|+\fracsm{42}{25}\varepsilon^B_1+\delta^B_1x_1\right]\times\\
&\qquad\exp\left\{\fracsm{4}{9}x_1^2+4x_1-16x_1(x_1-\fracsm{1}{2})^2\right\}.
\end{align*}
Thus
\begin{align*}
|V'(\fracsm{1}{2})-v'(\fracsm{1}{2})|\leq 171|U(0)-u(0)|+14\varepsilon_1^B+\fracsm{38}{9}\delta_1^B,\\
|V(\fracsm{1}{2})-v(\fracsm{1}{2})|\leq 31|U(0)-u(0)|+\fracsm{51}{20}\varepsilon_1^B+\fracsm{11}{5}\delta_1^B.
\end{align*}
Here, the last estimate followed by integration and using again $V(0)=v(0)$.

\underline{Bottom: $x_2$-graph to cylinder (Jacobi Equation)}

We consider the next part of $\Gamma$, graphical over the $x_2$-axis over $[a_0,\sqrt{2}]$. Here, we have the expressions:
\begin{align*}
&\left|\partial_2R(x_2,\xi,F')\right|=0,\\
&\left|\partial_3R(x_2,f,\xi')\right|=2\left|\frac{1}{x_2}-\frac{x_2}{2}\right||\xi'|,\\
&\left|\partial_2S(x_2,\xi,F')\right|=\left|\left(\frac{1}{x_2}-\frac{x_2}{2}\right)F'+\frac{\xi}{2}\right|,\\
&\left|\partial_3S(x_2,f,\xi')\right|=\left|\left(\frac{1}{x_2}-\frac{x_2}{2}\right)(x_2\xi'-f)+\frac{4\xi'}{x_2^2}\right|.
\end{align*}

Assume for a small $\delta_2^B>0$ the uniform bounds:
\begin{align}
&\left|G''+R(x_2,F,F')G'+S(x_2,F,F')G\right|\leq \delta_2^B,\\
&\left|\partial_3R(x_2,u,\xi')\right||g'|+\left|\partial_3S(x_2,f,\xi')\right||g|\leq \frac{4}{5}+8\left(x_2-\sqrt{2}\right)^2,\\
&\left|\partial_2 S(x_2,\xi,U')\right||g|\leq \frac{24}{50}-\frac{3}{4}\left(x_2-\sqrt{2}\right)^2.
\end{align}

We let $\Psi(x_2):=|G'(x_2)-g'(x_2)|$ and estimate:
\begin{align*}
\Psi'(x_2)&\leq \left||G'-g'|'\right|\overset{\textrm{a.e}}{=}|G''-g''|\\
&\leq \left|R(x_2,F,F')G' - R(x_2,f,f')g'\right|+\left|S(x_2,F,F')G- S(x_2,f,f')g\right|+\delta_2^B\\
&\leq \left|R(x_2,F,F')\right||G' -g'|+\left|R(x_2,F,F') - R(x_2,f,f')\right||g'|\\
&\qquad+\left|S(x_2,F,F')\right||G -g|+\left|S(x_2,F,F') - S(x_2,f,f')\right||g|+\delta_2^B\\
&\leq \left|R(x_2,F,F')\right||G' - g'|+\left|S(x_2,F,F')\right||G - g|\\
&\qquad +\left(\left|\partial_3R(x_2,f,\xi')\right||g'|+\left|\partial_3S(x_2,f,\xi')\right||g|\right)|F'-f'|\\
&\qquad +\left|\partial_2 S(x_2,\xi,F')\right||g||F-f|+\delta_2^B\\
&\leq \left|R(x_2,F,F')\right|\Psi(x_2) + \left|S(x_2,F,F')\right|\int_{0}^{x_2}\Psi(s)ds+\left|S(x_2,F,F')\right||G(a_0)-g(a_0)|\\
&\qquad+\left[\frac{4}{5}+8\left(x_2-\sqrt{2}\right)^2\right]|F'(x_2)-f'(x_2)|+\left[\frac{24}{50}-\frac{3}{4}\left(x_2-\sqrt{2}\right)^2\right]|F(x_2)-f(x_2)|+\delta_2^B.
\end{align*}
We integrate on $[a_0,x_2]$ to get:
\begin{align*}
\Psi(x_2)\leq &\int_{a_0}^{x_2}\left[\left|R(s,F,F')\right|+\int_{a_0}^{x_2}\left|S(t,F,F')\right|dt\right]\Psi(s)ds+\left(\int_{a_0}^{x_2}\left|S(s,F,F')\right|ds\right)|G(a_0)-g(a_0)|\\
&\quad +\int_{a_0}^{x_2}\left[\fracsm{4}{5}+8(s-\sqrt{2})^2\right]|F'(s)-f'(s)|ds+\int_{a_0}^{x_2}\left[\fracsm{24}{50}-\fracsm{3}{4}(s-\sqrt{2})^2\right]|F(s)-f(s)|ds\\
&\quad+|G'(a_0)-g'(a_0)|+\delta_2^B(x_2-a_0).
\end{align*}
Recall, we have above shown the estimates:
\begin{align*}
&|F'(x_2)-f'(x_2)|\leq \Big[(\fracsm{16}{25}x_2-\fracsm{3}{7})|F(a_0)-f(a_0)|+|F'(a_0)-f'(a_0)|+\varepsilon_2^B(x_2-a_0)\Big]\\
&\qquad\qquad\times\exp\left\{\fracsm{9}{10}-\fracsm{11}{10}\left(x_2-\fracsm{3}{2}\right)^2+\left(x_2-a_0\right)\left(\fracsm{16}{25}x_2-\fracsm{3}{7}\right)\right\},\\
&|F(a_0)-f(a_0)|\leq\fracsm{10}{12}\left(\fracsm{3}{5}\varepsilon^B_1+\fracsm{23}{5}|U(0)-u(0)|\right)\\
&|F'(a_0)-f'(a_0)|\leq\left(\fracsm{10}{12}\right)^2(\fracsm{28}{5}\varepsilon_1^B+39|U(0)-u(0)|).
\end{align*}

We therefore get the bound:
\begin{align*}
\int_{a_0}^{\sqrt{2}}\left[\fracsm{4}{5}+8(s-\sqrt{2})^2\right]|F'(s)-f'(s)|ds\leq 78|U(0)-u(0)| + 12\varepsilon_1^B +\fracsm{9}{10}\varepsilon_2^B.
\end{align*}

Again, we will need the sizes of some elementary Gaussian double integrals:
\begin{align*}
&\int_{a_0}^{\sqrt{2}}\left[\fracsm{24}{50}-\fracsm{3}{4}(x_2-\sqrt{2})^2\right]\left(1+\int_{a_0}^{x_2}(\fracsm{16}{25}s-\fracsm{3}{7})\exp\left\{\fracsm{9}{10}-\fracsm{11}{10}\left(s-\fracsm{3}{2}\right)^2+\left(s-a_0\right)\left(\fracsm{16}{25}s-\fracsm{3}{7}\right)\right\}ds\right)dx_2\\
&\qquad\leq\frac{3}{10},\\
&\int_{a_0}^{\sqrt{2}}\int_{a_0}^{x_2}\left[\fracsm{24}{50}-\fracsm{3}{4}(x_2-\sqrt{2})^2\right]\exp\left\{\fracsm{9}{10}-\fracsm{11}{10}\left(x_2-\fracsm{3}{2}\right)^2+\left(x_2-a_0\right)\left(\fracsm{16}{25}x_2-\fracsm{3}{7}\right)\right\}dsdx_2\leq\fracsm{1}{5},\\
&\int_{a_0}^{\sqrt{2}}\int_{a_0}^{x_2}\left[\fracsm{24}{50}-\fracsm{3}{4}(x_2-\sqrt{2})^2\right](s-a_0)\exp\left\{\fracsm{9}{10}-\fracsm{11}{10}\left(x_2-\fracsm{3}{2}\right)^2+\left(x_2-a_0\right)\left(\fracsm{16}{25}x_2-\fracsm{3}{7}\right)\right\}dsdx_2\leq \frac{7}{125}.
\end{align*}
Thus
\[
\int_{a_0}^{\sqrt{2}}\left[\fracsm{24}{50}-\fracsm{3}{4}(x_2-\sqrt{2})^2\right]|F(x_2)-f(x_2)|dx_2\leq \fracsm{197}{30}|U(0)-u(0)|+\fracsm{47}{50}\varepsilon_1^B+\fracsm{7}{125}\varepsilon_2^B.
\]

Assume again bounds for the $\varepsilon$-geodesics:
\begin{align}
&\int_{a_0}^{x_2}|R(s,F,F')|ds\leq \fracsm{9}{20}-\fracsm{3}{4}(x_2-\sqrt{2})^2,\\
&\int_{a_0}^{x_2}\left|S(s,F,F')\right|ds\leq x_2-\fracsm{2}{5}.
\end{align}

By Gr\"{o}nwall-Bellman with the new estimates, we see:
\begin{align*}
\Psi(x_2)\leq &\Bigg[\left(\int_{a_0}^{x_2}\left|S(s,F,F')\right|ds\right)|G(a_0)-g(a_0)|
+\int_{a_0}^{x_2}\left[\fracsm{4}{5}+8(s-\sqrt{2})^2\right]|F'(s)-f'(s)|ds\\
&\qquad+\int_{a_0}^{x_2}\left[\fracsm{24}{50}-\fracsm{3}{4}(x_2-\sqrt{2})^2\right]|F(s)-f(s)|ds+\delta^B_2(x_2-a_0)+|G'(a_0)-g'(a_0)|\Bigg]\times\\
&\exp\left\{\int_{a_0}^{x_2}\left|R(s,F,F')\right|ds+(x_2-a_0)\int_{a_0}^{x_2}\left|S(s,F,F')\right|ds\right\}\\
\leq & \fracsm{63}{20}\Big[31|U(0)-U(0)|+\fracsm{51}{20}\varepsilon_1^B+\fracsm{11}{5}\delta_1^B + 78|U(0)-u(0)| + 12\varepsilon_1^B +\fracsm{9}{10}\varepsilon_2^B+\fracsm{197}{30}|U(0)-u(0)|\\
&\qquad+\fracsm{47}{50}\varepsilon_1^B+\fracsm{7}{125}\varepsilon_2^B + (\sqrt{2}-a_0)\delta^B_2+\frac{171|U(0)-u(0)|+14\varepsilon_1^B+\fracsm{38}{9}\delta_1^B}{|u'(\fracsm{1}{2})|}\Big]\\ 
\end{align*}
Thus
\begin{align*}
&|G'(\fracsm{1}{2})-g'(\fracsm{1}{2})|\leq 813|U(0)-u(0)|+86\varepsilon_1^B+3\varepsilon_2^B+\fracsm{95}{27}\delta_1^B+ (\sqrt{2}-a_0)\delta^B_2,\\
&|G(\fracsm{1}{2})-g(\fracsm{1}{2})|\leq 522|U(0)-u(0)| + 56\varepsilon_1^B+2\varepsilon_2^B + 11\delta_1^B+ \fracsm{9}{10}\delta^B_2.
\end{align*}

We see that, with matched initial conditions, choosing all approximation constants $\varepsilon_i^{T/B}$ and $\delta_i^{T/B}$ at the order of $10^{-1}$--$10^{-3}$ (depending on each particular coefficient) suffices to conclude the estimates in (\ref{topmatrix})--(\ref{botmatrix}), and hence one may, with such test functions, rigorously verify that locations of the Jacobi functions are indeed accurately shown in Figures 3--6, such that the conclusions in Theorem \ref{TechThm} hold true. Hence the main claim, the existence and properties of the closed self-shrinkers in Theorem \ref{Thm:torusphere} is also verified.

\bibliographystyle{amsalpha}

\end{document}